\documentclass[10pt, reqno]{amsart} 
\usepackage{bigints}
\usepackage{todonotes}
\presetkeys{todonotes}{color=orange!30}{}
\usepackage{color}

\usepackage{amsthm} 
\usepackage{amsmath,amssymb} 
\usepackage{mathrsfs}

\usepackage{hyperref} 
\usepackage{varioref} 

\numberwithin{equation}{section}

\theoremstyle{plain} 
\newtheorem{theo}{Theorem}[section]
\newtheorem{prop}[theo]{Proposition}
\newtheorem{lemma}[theo]{Lemma}

\theoremstyle{definition}
\newtheorem{rem}[theo]{Remark}
\newtheorem{exa}[theo]{Example}

\newtheorem{definition}[theo]{Definition}

\DeclareMathAlphabet{\mathpzc}{OT1}{pzc}{m}{it} 

\newcommand{\be}{\end{eqnarray*}}
\newcommand{\ee}{\end{eqnarray*}}
\newcommand{\ben}{\begin{eqnarray}}
\newcommand{\een}{\end{eqnarray}}

\def \R{\mathbb R}
\def \N{\mathbb N}

\def \B{\mathbb B}
\def \eps{\varepsilon}

\def \U{\mathcal U}
\def \C{\mathcal C}
\def \RR{\mathcal R}
\def \A{\mathcal A}

\def \M{\mathcal M}
\def \NN{\mathcal N}
\def \T{\mathcal T}
\def \E{\mathcal E}

\def\W{\mathcal{W}}
\def\K{\mathcal{K}}

\def\S{\mathcal{S}}
\def\J{\mathcal{J}}

\def\h{\mathfrak h}
\def\g{\mathfrak g}

\def \Ll{\mathscr L}
\def \BB{\mathcal B}

\def \d{\textbf d}

\def\eps{\varepsilon}
\def \vsm{\vskip 0.2 truecm}
\def \vsmm{\vskip 0.1 truecm}
\def \ds{\displaystyle}

\def\bel{\begin{equation}\label}
\def\eeq{\end{equation}}
\def \w{\omega}

\def \weak{\rightharpoonup^*}
\begin{document}
\title[Impulsive optimal control problems with time delays]{Impulsive optimal control problems with time delays in the drift term}
 \thanks{This research is partially supported by the  INdAM-GNAMPA Project 2023, CUP E53C22001930001}
  \thanks{{\em Authors}. G. Fusco, Dipartimento di Matematica Tullio Levi-Civita,
Universit\`a di Padova, Via Trieste 63, Padova  35121, Italy. 
Email:\,
fusco@math.unipd.it \\
M. Motta, Dipartimento di Matematica Tullio Levi-Civita,
Universit\`a di Padova, Via Trieste 63, Padova  35121, Italy. 
Email:\,
motta@math.unipd.it}
 \thanks{$^*$ Corresponding author: Giovanni Fusco}
\author{Giovanni Fusco}
\author{Monica Motta}
%
 \date{\today}
\begin{abstract} 
We  introduce a  notion of bounded variation solution for a new class of  nonlinear control systems with ordinary and impulsive controls, in which the drift function depends not only on the state, but also on its past history, through a finite number of time delays. After proving the well-posedness of such solutions and the continuity of the corresponding input-output map with respect to suitable topologies, we establish  
 necessary  optimality conditions for an associated optimal control problem. The approach, which involves approximating the problem by a non-impulsive optimal control problem with time delays and using  Ekeland’s principle combined with a recent, nonsmooth version of the Maximum Principle for  conventional delayed systems,  allows us to deal with mild regularity assumptions and a general endpoint constraint.   
 \end{abstract}

\maketitle
\section{Introduction}
We  establish necessary optimality conditions, in the form of a nonsmooth Maximum Principle, for the following impulsive optimal control problem with time delays, labeled (P): 
\[
\text{Minimize} \,\,\,   \Phi(x(0),x(T)) + \int_0^T l_0\left(t,\{x(t-h_k)\}_{k=0}^N,\alpha(t)\right) dt + \int_{[0,T]} l_1(t, \alpha(t)) \mu(dt)
\]
over the set of control pairs $(\mu,\alpha)$ with  $\mu \in C^*(\K)$ and $\alpha \in  \A_\mu$, and bounded variation trajectories $x:[-h,T]\to\R^n$, satisfying the  control system  
\bel{ics1}
\begin{cases}
x(t)=x(0)+\int_0^t f\left(s,\{x(s-h_k)\}_{k=0}^N,\alpha(s)\right) ds + \int_{[0,t]}G(s, \alpha(s)) \mu(ds) \quad \text{$\forall t\in]0,T]$,}  \\
x(t)=\zeta(t), \quad \text{a.e. $t\in[-h,0[$,} 
\end{cases}
\eeq
and satisfying the endpoint constraint
\bel{endcon}
(x(0),x(T))\in\T.
\eeq
Problem (P) involves both measurable functions and vector-valued measures as controls, since, fixed a time horizon $T>0$,  $C^*(\K)$ denotes the set of   regular measures $\mu$ on the Borel subsets of $[0,T]$ with range belonging to a closed convex cone $\K\subset \R^m$,  and  $\A_\mu:= \{\alpha:[0,T]\to\R^q  \text{ measurable s.t. } \alpha(t)\in A(t)  \text{ a.e. and $\mu$-a.e.} \}$, where $A:[0,T]\rightsquigarrow \R^q$ is a set-valued map. 
 Furthermore, the data comprise  real numbers  $0=h_0< h_1<\dots<h_N=:h$,   functions $\Phi:\R^{2n}\to\R$, $l_0:[0,T]\times(\R^n)^{N+1}\times \R^q \to\R$, $l_1:[0,T]\times \R^q\to\R^m$, $f:[0,T]\times(\R^n)^{N+1}\times \R^q \to\R^n$, $G:[0,T]\times \R^q\to\R^{n\times m}$, $\zeta:[-h,0]\to\R^n$ (which describes the ``history" of the state trajectories before time $0$),  and a closed subset $\T\subset\R^{2n}$ (the target). Precise assumptions and definitions will be given in Sec. \ref{S1}. 
 
This problem can be interpreted as an extension of the following non-impulsive optimal control problem with time delays, labeled ${\rm (P')}$, in which  dynamics and cost  depend linearly  on an unbounded control $\w$, in addition   to an ordinary control $\alpha$:
\[
\text{Minimize} \,\,\,   \Phi(x(0),x(T)) + \int_0^T \left[l_0\left(t,\{x(t-h_k)\}_{k=0}^N,\alpha(t)\right) + l_1(t, \alpha(t)) \w(t)\right]dt
\]
over the set of controls  $(w,\alpha)\in  L^1([0,T],\R^m\times\R^q)$ such that $\w(t)\in\K$ and $\alpha(t)\in A(t)$ for a.e. $t\in[0,T]$, and absolutely continuous trajectories $x:[-h,T]\to\R^n$, satisfying the conventional control system with time delays 
\bel{ics2}
\begin{cases}
\dot x(t)=f\left(t,\{x(t-h_k)\}_{k=0}^N,\alpha(t)\right)+G(t, \alpha(t))\w(t), \quad t\in]0,T], \\
x(t)=\zeta(t) \quad \text{a.e. $t\in [-h,0[$,}
\end{cases}
\eeq
and the endpoint constraint \eqref{endcon}.
Indeed, given a non-impulsive control pair $(\w,\alpha)$ and a corresponding solution $x$  to \eqref{ics2},  
$(\w,\alpha)$  can be identified with the impulsive control $(\mu,\alpha)=(\w\,dt,\alpha)$ and   $x$ clearly satisfies \eqref{ics1} for such $(\mu,\alpha)$. 
\vsm
As first results, we prove that this dynamics' extension is proper and the notion of impulsive trajectory is robust.  Specifically,   under mild hypotheses   we establish: 
\vsm
{\em Well-posedness:} given an initial condition $\xi\in\R^n$ and a control pair  $(\mu,\alpha)$,  there is one and only one corresponding  bounded variation trajectory $x$ to \eqref{ics1} satisfying $x(0)=\xi$;
\vsm
{\em Density:}    for any $(\mu,\alpha)$ and $x$ as above,   there exists a sequence of non-impulsive controls $(\w_i,\alpha_i)$ and corresponding trajectories $x_i$ to \eqref{ics2},  such that $dx_i\weak d\bar x$ and  $x_i(t)\to \bar x(t)$ in a full measure subset of $[0,T]$ containing 0 and $T$. 
\vsm
{\em Continuity of the input-output map:} the map
$(\xi,\mu,\alpha)\mapsto  x$,
where $x$  is the solution to \eqref{ics1} associated with $(\mu,\alpha)$ and such that $ x(0)=\xi$, is  continuous in a suitable sense (see Thm. \ref{input_output}). 
\vsm
Afterwards, the main result of the paper is expressed as a  {\em Maximum Principle} for the impulsive optimal control problem with time delays,  (P). 
\vsm

We emphasize that these results are obtained without invoking a well-known reparameterization technique,   usually employed in impulsive control without time delays since early work \cite{Ris:65,War:65}. In particular, this technique  
leads  to the introduction of an auxiliary free end-time optimal control problem  with ordinary controls only, in which time is considered as a state variable (see also \cite{BR:88,Mi:94,MR:95}).   In fact,   this procedure   seems hardly extendable to the case with time delays,  since the auxiliary problem that one obtains is not standard, as it involves time delays depending on the control itself.
Furthermore, this method
would require  Lipschitz continuity of the data in their $t$-dependence and a fixed control set $A$ in place of the set-valued map $A(t)$, as $t$ plays the role of a state. We follow instead a different approach proposed in \cite{VP}, which involves approximating our impulsive optimal control problem with time delays by a conventional one,  without any time change. We point out that, as in the case without delays, this approach is applicable  because the Lagrangian cost function $l_1$ and the {\em fast dynamics} $G$ do not depend on the state variable. 

 \vsm
The Maximum Principle established in this paper extends, on the one hand, the impulsive  Maximum Principle obtained in \cite{VP} to time delayed systems, and, on the other hand, the Maximum Principle for time delayed  problems of \cite{VB} to the impulsive setting considered here. Furthermore, it  might also be useful for the applications' relevant problem of parameter identification, whenever ordinary controls in the system are treated as unknown parameters. Our primary goal here is  to provide the weakest hypotheses our methods permit under which necessary conditions can be formulated and proved for problem (P).  In particular, we do not address  the infimum gap phenomenon, that is, we do not look for conditions guaranteeing that  the infimum of the extended problem (P) coincides with that of ${\rm (P')}$.  Results of this kind for very general impulsive problems without delays can be found e.g. in \cite{AMR15,MRV,FM1,FM2,FM3,FM4,PR,MPR} and references therein.  

It is worth mentioning that in the literature there are  several results on the stabilizability of delayed impulsive control systems and on the optimization of some specific related problems, but they all concern the so-called   `impulse model', where impulsive controls essentially reduce to a finite or countable number of jump instants, with a preassigned jump-function. Therefore, the line of research starting with this paper, based on a different notion of impulsive control system, could have interesting implications for applications (e.g. to fed-batch fermentation \cite{XSSZ02,GLFX06} or to impulsive control of delayed neural networks \cite{LCH20}), as it would allow the development of  new nonlinear models, impulsive and with time delays, with a significant freer allocation of the impulses.

\vsm
The paper is organized as follows.   In  Sec. \ref{S1}  we introduce in a rigorous way the concepts of  control  and  trajectory for the delayed impulsive control system considered in problem (P) and prove some  fundamental properties of the set of the impulsive trajectories. In Sec. \ref{S2} we establish the Maximum Principle, whose proof in given in Sec. \ref{MP_proof}. An Appendix with some technical proofs concludes the paper.

\subsection{Notations and preliminaries}\label{sub1.1}
We write $\BB$ and $\Ll$ for the sets of the Borel subsets and the Lebesgue subsets of $[0,T]$, respectively, and we denote the set of Borel subsets of $\R^l$ by $\BB^l$. Given an interval $I\subseteq\R$ and a set $\Omega \subseteq \R^l$, we write $\M(I,\Omega)$, $C(I,\Omega)$, $W^{1,1}(I,\Omega)$, $BV(I,\Omega)$  for the space of measurable, continuous, absolutely continuous and bounded variation functions on $I$ and with values in $\Omega$,  respectively. We will use $\| \cdot \|_{L^{\infty}(I)}$ and $\| \cdot \|_{L^1(I)}$ to denote the ess-sup norm on $I$ and the $L^1$-norm on $I$, respectively. When the domain is clear, we will sometimes  simply write $\| \cdot \|_{L^{\infty}}$ and $\| \cdot \|_{L^1}$. 
We denote by $C^*(\Omega)$ the set of signed and regular measures $\mu:\BB\to \Omega$ (from now on we will refer to such $\mu$ simply as measures), and we set $C^\oplus:= C^*(\R_{\geq0})$, where $\R_{\geq 0} := [0,+\infty[$.  For all these classes of functions,  we will not specify domain and  codomain when the meaning is clear. Given $\mu\in C^* (\R^l)$, {\em $\mu$-a.e.} means ``almost everywhere w.r.t. $\mu$", and when we do not specify $\mu$ we implicitly refer to the Lebesgue measure.
It is known that there is a bijection between $C^*(\R^l)$ and the set of the equivalence classes of functions of bounded variation from $[0,T]$ to $\R^l$ which are right continuous on $]0,T[$ and differ for a constant. In particular, given $\mu\in C^*(\R^l)$, we can associate with it the discontinuous trajectory $z$ such that $dz(t) = \mu(dt)$, given by 
\bel{def_integrale}
z(t) =   \int_{[0,t]} dz(s) =   \int_{[0,t]} \mu(ds) \qquad \forall t\in]0,T].
\eeq
Given $\mu\in C^*(\R^l)$, the \textit{total variation measure} is the measure $|\mu|\in C^\oplus$ given by
$$
|\mu| = \sum_{j=1}^l (\mu^j_+ + \mu^j_-),
$$
where $\mu^j_+, \mu^j_-\in C^\oplus$ are the elements of the Jordan decomposition of the $j$-th component $\mu^j$ of $\mu$. 
In $C^*(\R^l)$, the norm is the sum of the total variations of the component measures, i.e. $\|\mu\|_{C^*(\R^l)} := \sum_{j=1}^l |\mu^j|([0,T])$. Note that the components of $\mu\in C^*(\R^l)$ are absolutely continuous with respect to the total variation measure $|\mu|$ (we write $\mu^j \ll |\mu|$ for any $j$), hence there exists a function 
$\w:[0,T]\to\R^l$, which we will refer to as the \textit{Radon-Nikodym derivative of $\mu$ with respect to $|\mu|$}, such that $\mu(dt)=\w(t) |\mu|(dt)$. Equivalently, we will write $\w=\frac{d\mu}{d|\mu|}$. A vector-valued function $\Psi:[0,T]\to\R^l$ is said to be $\mu$-integrable if $t\mapsto |\Psi(t)\cdot \w(t)|$ is integrable with respect to $|\mu|$, and in this case we set \footnote{This definition is more general than $\int_B \Psi(t) \cdot \mu(dt) = \sum_j \int_B \Psi_j(t) \mu^j(dt)$, since the latter requires that $\Psi_j$ is $\mu^j$-integrable for any $j$.}
\bel{def_integrals}
\int_B \Psi(t) \cdot \mu(dt) = \int_B\Psi(t)\cdot\w(t) |\mu|(dt), \qquad \forall B\in\BB. 
\eeq
Given a sequence $(\mu_i)_i\subset C^*(\R^l)$ and $\mu\in C^*(\R^l)$, we write $\mu_i\weak \mu$ if  
$$
\lim_i \int_{[0,T]} \varphi(t) \mu^j_i(dt) = \int_{[0,T]} \varphi(t) \mu^j(dt), \qquad \forall \varphi\in C([0,T], \R), \ \forall j=1,\dots,l.
$$ 

We denote by $\ell(\Omega)$, ${\rm co}(\Omega)$, $\overline{\Omega}$, and $\partial \Omega$ the Lebesgue measure, the convex hull,  the closure, and the boundary of $\Omega$, respectively. 
As is customary,  $\chi_{_\Omega}$ is the characteristic function of $\Omega$, namely $\chi_{_\Omega}(x)=1$ if $x\in \Omega$ and $\chi_{_\Omega}(x)=0$ if $x\in\R^l\setminus \Omega$.  
For any  $a,b \in \R$, we write $a \wedge b:= \min \{a,b\}$. Given $r>0$, we denote the closed ball of radius $r$ in $\R^l$ by $r\B_l$, omitting the dimension when it is clear from the context. 
Given a closed set $C \subseteq \R^l$  and a point $z \in \R^l$, we define the distance of $z$ from $C$ as $d_C(z) := \min_{y \in C} |z-y|$, and we define the {\em support function} $\sigma_C$ of the set $C$ as $\sigma_C (y) :=\sup \{ z\cdot y \text{ : } z\in C \}$ for any $y\in\R^l$. 

\vsm 
Some standard constructs from nonsmooth analysis are employed in this paper. For background material we refer the reader for instance to \cite{Cl,CLSW,OptV}. A set 
$\K \subseteq \R^l$ is a {\em cone} if $\alpha x \in \K$ for any $\alpha >0$,  whenever $x \in \K$. Take a closed set $C \subseteq \R^l$ and a point $\bar x \in C$, the \textit{limiting normal cone} $N_C(\bar x)$ of $C$ at $\bar x$ is defined by 
\[ 
N_C(\bar x) := \left\{  \xi  \text{ : } \exists x_i \stackrel{C}{\to} \bar x,\, \xi_i \to \xi \,\,\text{ such that }\,\,  \limsup_{x \to x_i}  \frac{\xi_i \cdot (x-x_i)}{|x-x_i|} \leq 0 \ \ \forall i \right\},
\]
 in which the notation $x_{i} \stackrel{C}{\longrightarrow}\bar{x}$ is used to indicate that  all points in the sequence  $(x_i)_i$  lay in $C$.  
Take a lower semicontinuous function  $H:\R^l \to \R$  and a point $\bar x \in \R^l$, the \textit{limiting subdifferential} of $H$ at $\bar x$ is 
\[
\partial H(\bar x) := \left\{     \xi \text{ : } \exists  \xi_i \to \xi, \, x_i \to \bar x \text{ s.t. } \limsup_{x \to x_i} \frac{\xi_i \cdot (x-x_i) - H(x) + H(x_i)}{|x-x_i|}  \leq 0   \ \forall i  \right\}.
\]
 If  $H:\R^l\times\R^{h} \to \R$ is a lower semicontinuous function and $(\bar x,\bar y) \in \R^l\times\R^{h}$, we write $\partial_x H(\bar x,\bar y)$ [resp. $\partial_y H(\bar x,\bar y)$] to denote the {\em partial limiting subdifferential of $H$ at  $(\bar x,\bar y)$ w.r.t. $x$} [resp. w.r.t. $y$], and we write $\tilde\partial_{x} H(\bar x, \bar y)$ [resp. $\tilde\partial_{y} H(\bar x, \bar y)$] to denote the {\em projected limiting subdifferential w.r.t. $x$} [resp. w.r.t. $y$], i.e. the projection of the limiting subdifferential of $H(\cdot,\cdot)$ at $(\bar x,\bar y)$ onto the $x$-coordinate [resp. $y$-coordinate]. 
Given a locally Lipschitz continuous function $G:\R^k\to\R^l$ and $\bar x \in \R^k$, we write $D G(\bar x)$ to denote the {\em Clarke generalized Jacobian of $G$ at $\bar x$}, defined as 
\[
D G(\bar x) :=\text{\,co\,} \, \left\{    \xi    \text{: } \exists  \   x_i \stackrel{\text{diff}(G) \setminus \{ \bar x\}}{\longrightarrow}  \bar x  \text{ and }   \nabla G(x_i) \to \xi       \right\},
\] 
where $\nabla G$ denotes the classical Jacobian matrix of $G$ and diff$(G)$ denotes the full measure set of differentiability points of $G$. We recall that the set-valued map $x\leadsto DG(x)$ has nonempty, compact, convex values  and is upper semicontinuous. 

\section{BV trajectories of delayed impulsive control systems}\label{S1} 
 In this section we introduce in a rigorous way the concepts of  control  and  trajectory for an impulsive control system with $N$ time delays $0<h_1<\dots<h_N$, of the form 
\bel{imp_del_sys}
\begin{cases} 
  dx(t)=f(t,x(t),x(t-h_1),\dots,x(t-h_N),\alpha(t))\,dt+G(t,\alpha(t))\mu(dt), \ \  t\in[0,T], \\
   x(t) = \zeta(t) \ \  \text{a.e. }t\in [-h,0[ \qquad(h=h_N),
\end{cases}
\eeq
 and establish the main properties of the set of the corresponding trajectories.

\subsection{Statements and main results} 
We define the set  $\U$ of {\em  impulsive controls} as follows:
\bel{adm_imp_c}
\U:=\left\{(\mu,\alpha): \  \ \mu\in C^*(\K), \ \ \alpha\in\A_\mu, \  t\mapsto G(t,\alpha(t)) \ \text{$\mu$-integrable}\right\},
\eeq 
where $$\A_\mu:= \{\alpha:[0,T]\to\R^q: \   \text{$\alpha$ is  Borel measurable and } \alpha(t)\in A(t)  \text{ a.e. and $\mu$-a.e.} \}.$$ A solution to \eqref{imp_del_sys} is defined as a bounded variation function $x:[-h,T]\to\R^n$ satisfying
 \bel{delay_sys}
\begin{cases}
\ds x(t)= x(0) +\int_0^t f\left(s,\{x(s-h_k)\},\alpha(s)\right)ds +  \int_{[0,t]}G(s, \alpha(s)) \mu(ds) \ \ \forall t\in]0,T],  \\
x(t) = \zeta(t) \quad \text{a.e. }t\in [-h,0[,
\end{cases}
\eeq 
where we write $\{x(t-h_k)\}$ in place of $\{x(t-h_k)\}_{k=0}^N :=(x(t),x(t-h_1),\dots,x(t-h_N))$ and use the notion  \eqref{def_integrals} of $\mu$-integrability. We will refer  to $x$ as  an {\em impulsive trajectory} (or simply as a trajectory)  associated with $(\mu,\alpha)$ and to the triple  $(\mu,\alpha,x)$ as an {\em impulsive process} (or simply as a process) for \eqref{imp_del_sys}.   
When the measure $\mu$ is absolutely continuous w.r.t. the Lebesgue measure $\ell$, so that there exists   $\w\in L^1([0,T],\K)$   such that $\mu(dt)=\w(t)\,dt$, the trajectory $x$ is absolutely continuous and the impulsive control system \eqref{imp_del_sys} becomes the following conventional control system with time delays,
\bel{reg_del_sys}
\begin{cases}
 \dot x(t)=f(t,\{x(s-h_k)\},\alpha(t))+G(t,\alpha(t))\w(t), \qquad  t\in[0,T], \\
  x(t) = \zeta(t) \qquad \text{a.e. } t\in [-h,0[ .
\end{cases}
\eeq
As in the case without delays, the impulsive control system \eqref{imp_del_sys} can be seen as an extension of \eqref{reg_del_sys}. For this reason, a process  $(\mu,\alpha,x)$ with $\mu\ll\ell$ will be referred to as a {\em strict sense process} for \eqref{imp_del_sys}. 
 
\vsm
We  shall consider  the following   hypotheses:  
\vsm
{\em  \begin{itemize} 
\item[{\bf (H1)}] The set-valued map $[0,T]\ni t\rightsquigarrow A(t)\subseteq\R^q$ has graph  {\rm Gr}$(A):=\{(t,a) \text{ : } a\in A(t)  \}$ which is   $\Ll\times\BB^q$-measurable. \footnote{$\Ll\times\BB^q$ denotes the product $\sigma$-algebra of $\Ll$ and $\BB^q$.}  
\vsmm
\item[{\bf (H2)}] The function $\zeta$ belongs to $L^\infty([-h,0],\R^n)$. For every $\{z_k\}\in(\R^n)^{N+1}$, the function $(t,a)\mapsto f(t, \{z_k\}, a)$ is $\Ll\times\BB^q$-measurable and,   for any $M>0$, there exists $L_M\in L^1([0,T], \R_{\geq 0})$ such that 
\bel{lipschitz_f}
\begin{array}{l} |f(t,\{ z_k\},a) - f(t,\{y_k\},a)| \leq L_M(t) |\{z_k -y_k\}|,  \\[1.0ex] 
\qquad\qquad\qquad\qquad \forall \{z_k\},\, \{y_k\} \in M\,\B_{n(N+1)}, \  \forall a\in A(t), \ \text{ a.e. } t\in[0,T].
\end{array}
\eeq
\vsmm
\item[{\bf (H3)}] There exists $c \in L^1([0,T], \R_{\geq 0})$ such that 
\bel{bound_f}
\sup_{a\in A(t)} |f(t, \{ z_k\},a)| \leq c(t) (1+ |\{z_k\}|) \qquad\text{ $\forall \{z_k\}\in(\R^n)^{N+1}$, \ a.e. $t\in[0,T]$.}
\eeq
\vsmm
\item[{\bf (H4)}] The set valued function $t\leadsto \overline{\rm co} \,\g(t,A(t),\tilde \K)$, $t\in[0,T]$,  is uniformly bounded and continuous with respect to the Hausdorff metric, where $\g:\R\times\R^q\times\R^m\to\R^n$ is the function defined by
\bel{tilde_g}
\ds \g(t,a,w) := \frac{G(t,a)\cdot w}{1+\sum_{i=1}^n |\sum_{j=1}^m g_{ij}(t,a) w^j|},
\eeq
with the $g_{ij}$'s being the components of $G$ and $\tilde \K := \big\{w\in \K: \ \   |w^j| \leq 1, \ j=1,\dots,m \big\}$.
\end{itemize}
}
%
%


\begin{rem}\label{RH4} Hypothesis (H4)   is a kind of continuity assumption on a compactification  of the fast dynamics  $G$.  As we will see in Prop. \ref{Th_equivalenza} below,  (H4) allows us  to replace  without loss of generality the original control system with an equivalent {\em auxiliary}  control system with time delays, still impulsive but  in which the vector-valued measure $\mu$ is replaced by a nonnegative and scalar measure $\nu$ and an ordinary control $\w\in L^1([0,T],\tilde \K)$, while  the possibly highly irregular and unbounded term ``$G(t,a)$" is replaced by  ``$\g(t,a,w)$". 
In the approximation results, where essentially we want to deduce from the convergence of suitable scalar auxiliary controls $\nu_i\weak\nu$, the convergence of the corresponding vector-valued measures $\mu_i\weak\mu$ (see Prop. \ref{prop_density} and Lemma \ref{lemma_dense} below), we replace assumption (H4) with the following stronger condition (H4)$^*$:
{\em 
\begin{itemize}
\item[{\bf (H4)$^*$}]  The set valued function $t\leadsto \overline{\rm co} \,(\h,\g)(t,A(t),\tilde \K)$, $t\in[0,T]$,  is uniformly bounded and continuous with respect to the Hausdorff metric, where $\g$ is as in \eqref{tilde_g} and $\h:\R\times\R^q\times\R^m\to\R^m$ is the function defined by
\bel{tilde_g0}
\ds\h(t,a,w) := \frac{w}{1+\sum_{i=1}^n |\sum_{j=1}^m g_{ij}(t,a) w^j|}.  
\eeq
\end{itemize}}

Both assumptions (H4) and (H4)$^*$ are automatically satisfied by autonomous systems and also by systems where the multifunction $t\leadsto G(t,A(t))$ is uniformly bounded on $[0,T]$ and continuous in the sense of Kuratowski. 
\end{rem}

Let us state some fundamental properties of impulsive trajectories. 
\begin{prop}[Well-posedness]\label{glob_ex} Assume hypotheses {\rm (H1)}--{\rm (H3)}. Then, for any  control $(\mu,\alpha)\in\U$ and any initial condition $\xi\in\R^n$,  there exists one and only one impulsive trajectory $x$ of \eqref{imp_del_sys} associated with $(\mu,\alpha)$ and such that $x(0)=\xi$.
\end{prop}
\begin{proof} Given a control $(\mu,\alpha)\in\U$ and an initial condition $\xi\in\R^n$, let us consider the following (non-impulsive) ODE with time delays
\bel{AC_sys}
\left\{\begin{array}{l}
\dot {\tilde x}(t)=\varphi(t, \{\tilde x(t-h_k)\}) \qquad \text{a.e. }t\in[0,T], \\[1.0ex]
\tilde x(t)=\zeta(t) \quad \text{a.e. } t\in[-h,0[, \ \ \tilde x(0)=\xi,
\end{array}\right.
\eeq
where $\varphi:[0,T]\times (\R^n)^{N+1}\to\R^n$ is given by
$$
\varphi(t,\{z_k\}):=f(t,\{z_k+\tilde x_0(t-h_k)\},\alpha(t)) \qquad \forall t\in[0,T],\ \forall \{z_k\}\in(\R^n)^{N+1},
$$
and $\tilde x_0:[-h,T]\to\R^n$ is defined as
$$
\tilde x_0(t):=\int_{[0,t]}G(s,\alpha(s))\cdot\mu(ds), \ \  t\in]0,T], \quad \tilde x_0(t):=0 \ \ \forall t\in[-h,0].
$$
In particular, $\tilde x_0$ is a BV function uniquely determined by the control pair $(\mu,\alpha)$, and the function $t\mapsto \varphi(t,\{z_k\})$ is $\Ll$-measurable for any $\{z_k\}\in(\R^n)^{N+1}$. 

\vsm
The proof is complete as soon as we show that \eqref{AC_sys} admits a unique solution $\tilde x$, since $x:=\tilde x+ \tilde x_0$ turns out to be the unique solution of \eqref{imp_del_sys} associated with $(\mu,\alpha)$ and such that $x(0)=\xi$. Indeed, if  there are two solutions $x$, $x'$, with   $x'\neq x$,  to  \eqref{imp_del_sys} associated with $(\mu,\alpha)$ and such that $x(0)=\xi$, then  it follows immediately   that \eqref{AC_sys}  also has the distinct solutions $\tilde x:=x  -\tilde x_0$ and $\tilde x' := x' -\tilde x_0$.

  To this aim,  we define  $M>0$ and $\tilde c\in L^1([0,T],\R_{\geq0})$,   as
\bel{Mtildec}
\begin{split}
&M:= \Big[ |\xi| + \big(1+(N+1) \| \zeta \|_{L^{\infty}(-h,0)}\big) \|\tilde c\|_{L^1(0,T)} \Big] e^{(N+1)\| \tilde c \|_{L^1(0,T)}}, \\ 
&\tilde c(t):= \big(1+(N+1)\|\tilde x_0 \|_{L^\infty(0,T)}\big) \,c(t),
\end{split}
\eeq
where $c$ is as in (H3). Hence,  we introduce a function $\eta\in C^\infty((\R^n)^{N+1}, [0,1])$ which is equal to 1 in $M \B_{n(N+1)}$, equal to 0 in $(\R^n)^{N+1}\setminus 2M \B_{n(N+1)}$, and such that $\|\nabla\eta(\{z_k\})\|_{L^\infty}\leq 1$. Using a standard truncation technique,  in place of \eqref{AC_sys}  we consider the following ODE with time delays
\bel{AC_sys_aux}
\left\{\begin{array}{l}
\dot {\tilde x}(t)=\tilde\varphi(t, \{\tilde x(t-h_k)\}) \qquad \text{a.e. }t\in[0,T], \\[1.0ex]
\tilde x(t)=\zeta(t) \ \ \text{a.e. }t\in[-h,0[, \ \ \tilde x(0)=\xi,
\end{array}\right.
\eeq
where, the function $\tilde \varphi:[0,T]\times (\R^n)^{N+1}\to\R^n$ is given by
\[
\tilde \varphi(t,\{z_k\}) := \eta(\{z_k\})  \varphi(t,\{z_k\}) \qquad\forall t\in[0,T],\ \forall \{z_k\}\in(\R^n)^{N+1}.
\]
By (H3), for any $t\in[0,T]$ and any $\{z_k\}\in(\R^n)^{N+1}$,  one has
\bel{stima_tildefi}
\begin{split}
|\tilde \varphi(t,\{z_k\})| &\leq |\varphi(t,\{z_k\})| = |f(t, \{z_k + \tilde x_0(t-h_k)\}, \alpha(t))| \\
&\leq c(t) \big(1+(N+1) \| \tilde x_0 \|_{L^\infty(0,T)} + |\{z_k\}|\big) \leq \tilde c(t) (1+|\{z_k\}|),
\end{split}
\eeq
so that the very definition of $\tilde \varphi$ implies
\bel{stima_tildefi2}
|\tilde \varphi(t,\{z_k\})| \leq \tilde c(t) (1+2M).
\eeq
Moreover,  
there exists a function  $\tilde L\in L^1([0,T],\R_{\geq0})$ such that 
\bel{glob_int_lipschitz}
|\tilde\varphi(t,\{z_k\}) - \tilde\varphi(t,\{y_k\})| \leq \tilde L(t) |\{z_k-y_k\}| \qquad \text{ a.e. }t, \ \forall \{z_k\},\, \{y_k\}\in(\R^n)^{N+1}.
\eeq
Indeed, the  inequality $|\tilde \varphi(t,\{z_k\})-\tilde \varphi(t,\{y_k\})|$ is  zero when  both $\{z_k\}$, $\{y_k\}\in (\R^n)^{N+1}\setminus 2M\B_{n(N+1)}$, while, if, e.g.  $\{y_k\}\in (\R^n)^{N+1}\setminus 2M\B_{n(N+1)}$, so that $\eta(\{y_k\})=0$,  but $\{z_k\}\in  2M\B_{n(N+1)}$, we have
\[
\begin{split}
|\tilde\varphi(t,\{z_k\}) - \tilde\varphi(t,\{y_k\})| &= |\eta(\{z_k\})| \,|\varphi(t,\{z_k\})| \\
&=|\eta(\{z_k\})-\eta(\{y_k\})| \,|f(t, \{z_k+ \tilde x_0(t-h_k)\},\alpha(t))| \\
& \leq  c(t) \big(1+2M +(N+1) \| \tilde x_0\|_{L^\infty(0,T)}\big) \,|\{z_k\}-\{y_k\}|.
\end{split}
\]
If instead both $\{z_k\}$,  $\{y_k\}\in  2M\B_{n(N+1)}$,   we get
\[
\begin{split}
&|\tilde\varphi(t,\{z_k\}) - \tilde\varphi(t,\{y_k\})| \\
&\qquad\leq |\eta(\{z_k\})|\,|f(t, \{z_k+ \tilde x_0(t-h_k)\},\alpha(t)) -f(t, \{y_k + \tilde x_0(t-h_k)\},\alpha(t))| \\
&\qquad\qquad\qquad + |\eta(\{z_k\})-\eta(\{y_k\})|\, |f(t, \{y_k+ \tilde x_0(t-h_k)\},\alpha(t))|  \\
&\qquad\leq  L_{2M}(t) \,|\{z_k\}-\{y_k\}|+ c(t) \big(1+ 2M+ (N+1)\| \tilde x_0\|_{L^\infty(0,T)}\big)\, |\{z_k\}-\{y_k\}|,
\end{split}
\]
where $L_{2M}$ is as in (H2), for $M$ as in \eqref{Mtildec}.
Hence, \eqref{glob_int_lipschitz} holds with $\tilde L(t):= L_{2M}(t) +  c(t) \big(1+ 2M+ (N+1)\| \tilde x_0\|_{L^\infty(0,T)}\big)$.  At this point,    \eqref{AC_sys_aux} has a unique solution $\tilde x$ by \cite[Thm. 4.1]{VB}.

 In order to conclude, it remains to show that $\|\tilde x\|_{L^\infty(0,T)}\leq M$. From \eqref{stima_tildefi} it follows that
\[
|\tilde x(t)| \leq |\xi| + \int_0^t |\tilde\varphi(s, \{\tilde x(s-h_k)\})|\,ds \leq |\xi| + \int_0^t \tilde c(s) (1+|\{\tilde x(s-h_k)\}|) \, ds 
\]
for all $t\in[0,T]$, so that we obtain
\[
\|\tilde x\|_{L^\infty(0,t)}\leq |\xi| + \int_0^t \tilde c(s) \big[1+(N+1)\big(\|\tilde x\|_{L^\infty(0,s)} + \|\zeta\|_{L^\infty(-h,0)}\big)  \big]\, ds.
\]
A straightforward application of the Gronwall Lemma to the function $t\mapsto \|\tilde x\|_{L^\infty(0,t)}$, and the very definition of $M$ in \eqref{Mtildec}, imply that $\|\tilde x\|_{L^\infty(0,T)}\leq M$. Hence,  $\tilde x$ is the unique solution to \eqref{AC_sys} and $x=\tilde x+\tilde x_0$ is the unique solution to \eqref{imp_del_sys} associated with $(\mu,\alpha)$ and such that $x(0)=\xi$. The proof is thus complete.
\end{proof}
Next proposition establishes that  the set of strict sense trajectories is dense, in a suitable sense,  in the set  of  impulsive, BV trajectories.
 \begin{prop}[Density]\label{prop_density}
 Let hypotheses {\rm (H1)}--{\rm (H3)} and {\rm (H4)$^*$} be satisfied. Then, for any impulsive process  $(\bar \mu,\bar \alpha,\bar x)$,  there exists a sequence of strict sense processes $(\mu_i,\alpha_i,x_i)_i$, such that
$$
\mu_i \weak \bar\mu, \quad \ell\big(\{t\in[0,T]: \ \ \alpha_i(t)=\bar\alpha(t)\}\big)\to T,\quad dx_i \weak d\bar x,  \quad x_i(t)\to \bar x(t) \ \ \forall t\in\E,
$$  
where $\E$ is a full measure subset of $[0,T]$ which contains $0$ and $T$. 
 \end{prop}

Finally, under an additional hypothesis, we  obtain the following continuity property of the input-output map.
\begin{theo}[Continuity of the input-output map]\label{input_output}
Assume hypotheses  {\rm (H1)}--{\rm (H3)}, let $G$ be independent of $a$, i.e.  $G(t,a)=\hat G(t)$, and let  $\hat G$ be continuous on $[0,T]$.
Then, the input-output map      ${\mathcal I}: \R^n\times\U\mapsto BV([-h,T],\R^n),$  given by
\bel{input-output_map}
(\bar\xi,\bar\mu,\bar\alpha)\mapsto {\mathcal I}(\bar\xi,\bar\mu,\bar\alpha)= \bar x,
\eeq
where $\bar x$  is the impulsive trajectory associated with the control $(\bar\mu,\bar\alpha)$ and such that $\bar x(0)=\bar\xi$, is  (well-defined and) continuous, in the sense that  all sequences  $(\xi_i,\mu_i,\alpha_i)_i\subset \R^n\times\U$ that satisfy 
 \bel{close_i}
 \xi_i\to\bar \xi, \quad (\mu_i,|\mu_i|)\weak(\bar \mu,\lambda) \ \ \text{for some $\lambda\in C^\oplus$,} \quad   \ell(\{t\in[0,T] \text{ : } \alpha_i(t)= \bar \alpha(t) \})\to T,
\eeq
have corresponding trajectories $x_i:= {\mathcal I}(\xi_i,\mu_i,\alpha_i)$  that satisfies
\bel{close_x}
dx_i\weak d\bar x, \qquad x_i(t)\to \bar x(t)  \quad \forall t\in\E
\eeq
where $\E$ is a full measure subset of $[0,T]$ which contains $0$ and $T$.
\end{theo}

The proofs of Prop. \ref{prop_density} and Thm. \ref{input_output}, which require some preliminary results,  will be given in the next subsection. 

\vsm
The additional assumptions that $G$ does not depend on $a$ and is continuous   are crucial for the validity of Thm. \ref{input_output}, as shown by the following simple,  delay-free examples. 
\begin{exa}
Consider $T=1$, $G(t,a)=a$, $A(t)\equiv [0,1]$, $\K=\R_{\geq 0}$, and $f\equiv0$. Let $\xi_i\equiv\xi_0=0$, $\mu_i\equiv\mu=\delta_{\{0\}}$ be the Dirac unit measure concentrated at $0$, $\alpha(t)\equiv 0$, and let $(\alpha_i)_i\subset \A$ be  defined by 
\[
\alpha_i(t):=
\begin{cases}
1 \qquad t\in\left[0,\frac{1}{i}\right] \\
0 \qquad \text{otherwise}.
\end{cases}
\]
Of course, $\mu_i\weak \mu$ and $\ell(\{t \text{ : } \alpha_i(t) \neq \alpha(t) \}) =\frac{1}{i} \to 0$, but $dx_i(t)=\alpha_i(t)\mu_i(dt)$ does not weakly$^*$ converge to $dx(t)=0$ in $C^*([0,1])$. Indeed, take any $\varphi\in C([0,1], \R)$ with $\varphi(0)\neq 0$, then $\int_{[0,1]} \varphi(t)\alpha_i(t) \mu_i(dt) = \varphi(0) \alpha_i(0)=\varphi(0)$ for any $i$, while $\int_{[0,1]} \varphi(t) dx(t)=0$. Furthermore, for any $t>0$, we have $x(t)=0$, while $x_i(t)=1$ for any $i>1/t$, so that  $x_i(t)\to 1\ne  x(t)$. 
\end{exa}
\begin{exa}
Let $T=1$, $G(t)=\chi_{]0,1]}(t)$, $\mu=\delta_{\{0\}}$, $f=0$, and $\mu_i(dt) = m_i(t) dt$, where $m_i(t)$ is given by
\[
m_i(t):= 
\begin{cases}
i \qquad  \forall t\in\left[0,\frac{1}{i}\right],\\
0 \qquad\text{otherwise.}
\end{cases}
\]
Take $\varphi\in C([0,1])$. By the mean value theorem, for any $i$ there exists $t_i\in\left[ 0,\frac{1}{i}\right]$ such that
\[
\lim_i \int_{[0,1]} \varphi(t) \mu_i(dt)= \lim_i \ i \int_0^{\frac{1}{i}}\varphi(t)=\lim_i\varphi(t_i) = \varphi(0)=\int_{[0,1]} \varphi(t)\mu(dt).
\]
Thus,   $\mu_i\weak \mu$. However,  we have that $\int_{[0,1]} G(t)\mu(dt)=0$  but
$\int_{[0,1]} G(t) \mu_i(dt)=\int_0^1 m_i(t)dt = 1$ for all  $i$. 
Therefore, choosing   $\varphi\equiv1$, we see that $dx_i(t) = G(t)\mu_i(dt)$ does not weakly$^*$ converge in $C^*([0,1])$ to $dx(t)=G(t)\mu(dt)$. Similarly to the previous example, for any $t>0$,   $x(t)=0$   and  $x_i(t)=1$ for any $i>1/t$, so that we still have  $x_i(t)\to 1\ne  x(t)$. 

\end{exa}

%

\subsection{Proofs of Prop. \ref{prop_density} and Thm. \ref{input_output}}
In these proofs, equivalence between the given impulsive system and an auxiliary impulsive system,  and some   approximation results for measures play a key role. So, let us start with these preliminary results. 
\vsm
Consider the set of  {\em auxiliary controls} $\tilde\U$,  given by   
  \bel{adm_aux_c}
\tilde\U:=\left\{(\nu,\alpha,\w):  \ \nu\in C^\oplus,  \ \alpha\in\A_\nu, \ \  \w\in\W_\nu, \ \  t\mapsto \g(t,\alpha(t),\w(t)) \ \text{$\nu$-integrable} \right\},
\eeq 
where $\g$ is as in \eqref{tilde_g} and, for $\tilde\K$ as in (H4), $\W_\nu$ is given by
$$\W_\nu:= \{\w:[0,T]\to\R^m, \   \text{$\w$ is Borel  measurable, }  \ \w(t)\in \tilde \K   \text{ a.e. and $\nu$-a.e.}\}.$$
We introduce the following {\em auxiliary impulsive control system} with time delays,
\bel{aux_sys_eq}
\begin{cases}
dx(t)=  f\left(t,\{x(t-h_k)\}, \alpha(t)\right)\,dt +\g(t, \alpha(t),\w(t)) \nu(dt), \ \  t\in[0,T],  \\[1.5ex]
x(t) = \zeta(t) \ \ \text{a.e. } t\in [-h,0[,
\end{cases}
\eeq 
Precisely, for any auxiliary control $(\nu,\alpha,\w)$, we call {\em auxiliary trajectory} any bounded variation function $x$,  satisfying
\bel{aux_sys}
\begin{cases}
\ds x(t)= x(0)+\int_0^t f\left(s,\{x(s-h_k)\}, \alpha(s)\right)ds + \int_{[0,t]} \g(s, \alpha(s),\w(s)) \nu(ds) \  \forall t\in]0,T],  \\[1.5ex]
x(t) = \zeta(t) \ \ \text{a.e. } t\in [-h,0[,
\end{cases}
\eeq
and   refer to $(\nu,\alpha,\w,x)$ as  an  {\it auxiliary process}.  When $\nu\ll\ell$, we say that  $(\nu,\alpha,\w,x)$ is a {\it strict sense auxiliary process}, and we call $x$  a {\it strict sense auxiliary trajectory} and  $(\nu,\alpha,\w)$  a {\it strict sense auxiliary control}.

\begin{prop} \label{Th_equivalenza}
There is a one-to-one correspondence between processes for  \eqref{imp_del_sys} and auxiliary processes.  Precisely, we have that
\begin{itemize}
\item[(i)] given a process $(\mu,\alpha,x)$ for \eqref{imp_del_sys}, the four-tuple $(\nu,\alpha,\w,x)$, where   
\bel{nu}
\w=\frac{d\mu}{d|\mu|}, \qquad \nu(dt) := \big(1+{\sum}_{i=1}^n \big|{\sum}_{j=1}^m g_{ij}(t,\alpha(t)) \w^j(t) \big|\big) |\mu|(dt),
\eeq
 is an auxiliary process;
 \item[(ii)] given an auxiliary process $(\nu,\alpha,\w,x)$, then the triple  $(\mu,\alpha,x)$, where $\mu$ is   given by
\bel{mu}
\mu(dt) := \frac{\w(t) \nu(dt)}{1+{\sum}_{i=1}^n \big|{\sum}_{j=1}^m g_{ij}(t,\alpha(t)) \w^j(t) \big|},
\eeq
is a process for \eqref{imp_del_sys}.
\end{itemize}
In addition, a process  $(\mu,\alpha,x)$ is strict sense if and only if the corresponding auxiliary process  $(\nu,\alpha,\w,x)$ is strict sense.
\end{prop}

\begin{proof} The proof follows the same lines as the proof of  \cite[Lemma 6.1]{VP}, so we illustrate only the key points, omitting  details.
   Concerning statement (i),  given a process $(\mu,\alpha,x)$ we consider the control triple $(\nu,\alpha,\w)$, where $\w$ and $\nu$ are defined as above. Note that $\big(1+{\sum}_{i=1}^n \big|{\sum}_{j=1}^m g_{ij}(t,\alpha(t)) \w^j(t) \big|\big)$  is $|\mu|$ integrable in view of the definition of integrability that we have adopted; consequently $\nu$ is well defined. From \eqref{def_integrals} it follows that the function $t\mapsto \g(t,\alpha(t),\w(t))$ is $\nu$-integrable and 
\bel{uguaglianza_integrale}
\int_B G(t,\alpha(t))\mu(dt) = \int_B \g(t,\alpha(t),\w(t)) \nu(dt) \qquad \forall B\in\BB.
\eeq
Since clearly $\nu\ll|\mu|$ (in fact, $\A_\mu\subseteq \A_\nu$), this implies that  $(\nu,\alpha,\w)\in \tilde\U$.
As a direct consequence of \eqref{uguaglianza_integrale} we deduce that $(\nu,\alpha,\w,x)$ is an auxiliary process. Furthermore, when $(\mu,\alpha,x)$ is a strict sense process, namely $\mu\ll\ell$, one also has $|\mu|\ll\ell$, so that   $\nu\ll |\mu|$ implies   $\nu\ll \ell$ and  $(\nu,\alpha,\w,x)$ is strict sense too. 

Conversely, take an auxiliary process $(\nu,\alpha,\w,x)$ and consider the control pair $(\mu,\alpha)$, in which   $\mu$ is defined by 
\eqref{mu}.  It is straightforward that $\mu\ll\nu$.  Moreover, the function $t\mapsto G(t,\alpha(t))$ is $\mu$-integrable and \eqref{uguaglianza_integrale} still holds. Therefore,    $(\mu,\alpha)\in\U$ and $(\mu,\alpha,x)$ is a process  for \eqref{imp_del_sys}. Finally, the  relation $\mu\ll \nu$ implies that  $(\mu,\alpha,x)$ is a strict sense process as soon as $(\nu,\alpha,\w,x)$ is. 
 \end{proof}

%
%
%
We will   also use the following three lemmas. 
\begin{lemma} \label{lemma_convergenza}
Consider some sequences  $(\gamma_i)_i\subset C^*(\R^n)$,  $(\xi_i)_i\subset\R^n$,    $(\varphi_i)_i\subset\M([0,T]\times(\R^n)^{N+1},\R^n)$,  and some $\gamma_0\in C^*(\R^n)$, $\xi_0\in\R^n$,  $\varphi_0\in\M([0,T]\times(\R^n)^{N+1},\R^n)$, and  $\eta\in L^\infty([-h,0],\R^n)$ satisfying  the following conditions {\rm (i)}--{\rm (iv)}:  
\begin{itemize} 
\item[{\rm (i)}]   For any $M>0$,  there exists $\tilde L_M\in L^1([0,T],\R_{\geq 0})$ such that, for each $i\ge 1$,
\bel{lipschitz_phi}
|\varphi_i(t,\{z_k\}) - \varphi_i(t,\{y_k\})| \leq \tilde L_M(t) |\{z_k\} - \{y_k\}|, 
\eeq
 for all $\{z_k\}$,\, $\{y_k\} \in M\,\B_{n(N+1)}$ and for   a.e. $t\in[0,T]$.
Moreover, one has
\bel{conv_phi}
\ell(\{ t\in[0,T]\text{ : } \varphi_i(t,\{z_k\}) = \varphi_0(t,\{z_k\}) \ \  \forall \{z_k\}\in (\R^n)^{N+1} \}) \to T.
\eeq
\item[\rm{(ii)}]Given a function $z_0\in BV([-h,T],\R^n)$ satisfying  
\[
\begin{cases}
z_0(t) = \xi_0 +\int_0^t \varphi_0(s, \{z_0(s-h_k)\})ds + \int_{[0,t]}\gamma_0(ds) \qquad \forall t\in ]0,T], \\
z_0(t) = \eta(t) \ \  \text{a.e. }  t\in[-h,0[, \ \ \ z_0(0)=\xi_0,
\end{cases}
\]
 there exists $\tilde c\in L^1([0,T],\R_{\geq 0})$ such that, for each $i\ge 1$, 
\bel{bound_phi}
|\varphi_i(t,\{ z_0(t-h_k) \})|\leq \tilde c(t) \qquad \text{a.e. }t\in[0,T].
\eeq

\item[{\rm (iii)}]  The sequence  $(\gamma_i)_i$  is uniformly bounded in total variation and there exists a Borel subset  $\E\subset [0,T]$ with $\ell(\E)=T$ such that
\[
\lim_i \int_{[0,t]}  \gamma_i(ds) = \int_{[0,t]} \gamma_0(ds) \qquad \forall t\in\E.
\]
\item[{\rm (iv)}] The sequence $\xi_i\to\xi_0$.
\end{itemize}
Then, for any integer $i$  large enough there exists   $z_i\in BV([-h,T],\R^n)$,   that satisfies
\bel{eq_z}
\begin{cases}
\ds z_i(t) =\xi_i+ \int_0^t \varphi_i(s, \{z_i(s-h_k)\})ds + \int_{[0,t]}\gamma_i(ds) \qquad \forall t\in ]0,T], \\
z_i(t) = \eta(t) \ \  \text{a.e. }t\in[-h,0[, \ \ \ z_i(0)=\xi_i.
\end{cases}
\eeq
Moreover, $dz_i \weak dz_0$ and $z_i(t)-\int_{[0,t]} \gamma_i(ds) \to z_0(t) - \int_{[0,t]} \gamma_0(ds)$ uniformly in $[0,T]$, so that
$z_i(t)\to z(t)$ for all $t\in\E$.
\end{lemma}   
The proof of  Lemma  \ref{lemma_convergenza},     extending to time delayed systems the results of  \cite[Prop. 5.1]{VP},      will be given in the Appendix.
 
\begin{lemma}  \label{lemma_misure}
Let $\gamma\in C^*(\R^k)$ and $(\gamma_i)_i\subset C^*(\R^k)$. Then, the following properties hold true. 
\begin{itemize}
\item[(i)] If $(\gamma_i)_i$ is a uniformly bounded sequence in total variation and there exists a Borel subset $\E\subset [0,T]$ with $\ell(\E)=T$, $T\in\E$, and such that
\[
\lim_i \int_{[0,t]} \gamma_i(ds) = \int_{[0,t]} \gamma(ds) \qquad \forall t\in\E,
\]
then $\gamma_i\weak\gamma$.

\item[(ii)]  If $(\gamma_i,|\gamma_i|)\weak(\gamma,\lambda)$ for some $\lambda\in C^\oplus$, then there exists  a Borel subset $\E\subset[0,T]$ containing $T$, such that $[0,T]\setminus\E$ is at most countable and 
\[
\lim_i \int_{[0,t]} \Psi(s)\cdot\gamma_{i}(ds) = \int_{[0,t]} \Psi(s)\cdot\gamma(ds) \qquad \forall t\in\E,
\]
for any integer $n\ge1$ and every continuous function $\Psi:[0,T]\to\R^{n\times k}$. 
\item[(iii)] If $\gamma_i\weak\gamma$, there exists  a subsequence $(\gamma_{i_l})_l$ of $(\gamma_i)_i$ such that  $(\gamma_{i_l},|\gamma_{i_l}|)\weak(\gamma,\lambda)$ for some $\lambda\in C^\oplus$. 
\end{itemize}
\end{lemma}
Statement (i)   coincides with \cite[Prop.\,5.2\,(a)]{VP}, while   (ii) and (iii) slightly extend the results of \cite[Prop.\,5.2\,(b)]{VP}.  We will then provide   a concise proof of (ii), (iii) in the Appendix. 

\vsm
To state the third lemma,  let us define the sets  $\A$, $\W$,  as
\[
\A:= \{\alpha\in\M([0,T],\R^q) \text{ : } \alpha(t)\in A(t) \ \text{a.e.} \}, \quad \W:= \M([0,T],\tilde\K). 
\]

\begin{lemma}\cite[Prop.\,5.3]{VP}  \label{lemma_dense}
Assume {\rm (H1), (H4)}, and let $(\nu,\alpha,\w)$ be an auxiliary control.  Then, there exist sequences $(m_i)_i\subset L^1([0,T],\R_{\geq 0})$ and $(\alpha_i,\w_i)_i\subset \A\times\W$ such that
$$
m_i(t)dt \weak \nu(dt), \quad \ell(\{t\in[0,T] \text{ : } (\alpha_i,\w_i)(t)\neq(\alpha,\w)(t)\})\to0.
$$
Moreover, $t\mapsto\g(t,\alpha_i(t),\w_i(t))  m_i(t)$  is Lebesgue integrable for any $i$ and 
$$
\g(t,\alpha_i(t),\w_i(t))  m_i(t)dt \weak\g(t,\alpha(t),\w(t))  \nu(dt).
$$
If  {\rm (H4)} is replaced with the stronger assumption  {\rm (H4)$^*$}, then the function $t\mapsto (\g,\h)(t,\alpha_i(t),\w_i(t))  m_i(t)$ is Lebesgue integrable for any $i$ and 
$$
(\g,\h)(t,\alpha_i(t),\w_i(t))  m_i(t)dt \weak (\g,\h)(t,\alpha(t),\w(t))  \nu(dt).
$$
\end{lemma}

\begin{proof}[Proof of Prop. \ref{prop_density}]

Given an impulsive process $(\bar\mu,\bar\alpha,\bar x)$ as in the statement, let us consider the corresponding auxiliary process $(\bar\nu,\bar\alpha,\bar\w, \bar x)$,  as defined in Thm. \ref{Th_equivalenza}.  From  Lemma \ref{lemma_dense} applied to the auxiliary control $(\bar\nu,\bar\alpha,\bar\w)$,   it follows that there exists a sequence   $(\nu_i,\alpha_i,\w_i)_i\subset C^\oplus\times\A\times\W$ such that  $\nu_i(dt)=m_i(t)dt$ for some $m_i\in L^1([0,T],\R_{\geq0})$ for any  $i$,     
  $m_i(t)dt \weak \bar\nu(dt)$, $\ell(\{t\in[0,T] \text{ : } (\alpha_i,\w_i)(t)\neq(\bar\alpha,\bar\w)(t)\})\to0$, 
 $t\mapsto (\g,\h)(t,\alpha_i(t),\w_i(t))m_i(t)$ is Lebesgue integrable for any $i$,  and 
$(\g,\h)(t,\alpha_i(t),\w_i(t))m_i(t)dt \weak (\g,\h)(t,\bar\alpha(t),\bar\w(t))\bar\nu(dt).$ Notice that, for any $i$, $(\nu_i,\alpha_i,\w_i)$ is a strict sense  auxiliary control, since $\nu_i\ll\ell$.    Recalling that $\bar\mu(dt)=\h(t,\bar \alpha(t),\bar\w(t))\bar \nu(dt)$,  if we set, for any $i$,  
$$
\mu_i(dt):=\h(t,\alpha_i(t),\w_i(t))m_i(t)\,dt, \quad \gamma_i(dt):=\g(t,\alpha_i(t),\w_i(t))m_i(t)\,dt
$$
then $(\mu_i,\alpha_i)_i\subset \U$ and we have
$$
\mu_i\weak \bar\mu,\qquad \gamma_i\weak \gamma_0:=\g(t,\bar\alpha(t),\bar\w(t))\bar\nu(dt)=G(t,\bar\alpha(t))\cdot\bar\mu(dt). 
$$
From
Lemma \ref{lemma_misure},(ii)-(iii), with $\Psi$ identity matrix of $\R^{m\times m}$,  passing possibly  to a subsequence (we do not relabel),  it follows that there exists  a Borel subset $\E\subset[0,T]$ containing $T$,   such that $[0,T]\setminus\E$ is at most countable,  and 
\[
\lim_i \int_{[0,t]}\gamma_{i}(ds) = \int_{[0,t]}\gamma_0(ds) \qquad \forall t\in\E.
\]
Then, 
a straightforward application of Lemma \ref{lemma_convergenza} in which, in particular, $z_0 \equiv\bar x$ and $\varphi_i(t,\{z_k\}):=f(t,\{z_k\}, \alpha_i(t))$, so that (H3) implies $|\varphi_i(t,\{ \bar x(t-h_k) \})|\leq c(t)(1+ \|\bar x\|_{L^\infty})=: \tilde c(t)\in L^1([0,T],\R_{\geq0})$,
 leads us to deduce that, for any $i$ sufficiently large, there exists a strict sense trajectory $x_i$  associated with the control $(\mu_i,\alpha_i)$, for which  $x_i(0)= \bar x(0)$, $dx_i\weak d\bar x$, and $x_i(t)\to \bar x(t)$ on $\E$.  
\end{proof}

\begin{proof}[Proof of Thm. \ref{input_output}]
Let $(\bar\xi,\bar\mu,\bar\alpha)\in\R^n\times\U$, let $\bar x={\mathcal I}(\bar\xi,\bar\mu,\bar\alpha)$ (of course ${\mathcal I}$ is well defined in view of Prop. \ref{glob_ex}) and let $(\xi_i,\alpha_i,\mu_i)_i\subset\R^n\times\U$ satisfy the convergence conditions in \eqref{close_i}.  Since the function $t\mapsto \hat G(t)$ is continuous and  $(\mu_i, |\mu_i|)\weak(\bar \mu,\lambda)$, from Lemma \ref{lemma_misure}, (ii)   it follows that   there exists  a Borel subset $\E\subset[0,T]$ containing $T$ such that $[0,T]\setminus\E$ is at most countable and 
 \[
\lim_i \int_{[0,t]} \hat G(s)\cdot\mu_{i}(ds) = \int_{[0,t]}\hat G(s)\cdot\bar\mu(ds) \qquad \forall t\in\E.
\]
At this point, taking $\varphi_i(t,\{z_k\})= f(t,\{z_k\},\alpha_i(t))$, $\varphi_0(t,\{z_k\})=f(t,\{z_k\},\bar\alpha(t))$, $\gamma_i(dt)= \hat G(t)\mu_i(dt)$, $z_0\equiv \bar x$, and $\gamma_0(dt)=\hat G(t)\bar\mu(dt)$ (again, $|f(t,\{ \bar x(t-h_k) \},\alpha_i(t))|\leq c(t)(1+ \|\bar x\|_{L^\infty})\in L^1([0,T],\R_{\geq0})$), we can apply Lemma \ref{lemma_convergenza} and conclude that $dx_i\weak d\bar x$ and  $\lim_i x_i(t)=\bar x(t)$ for any $t\in\E$. 
\end{proof}

\begin{rem}
 Actually, the set $\E$ of Thm. \ref{input_output}  is exactly the set of continuity points of $\lambda$, weak limit of $(|\mu_i|)_i$. In particular, when the cone $\K$ coincides  the first orthant, $\E$ is   the set of continuity points of  $|\mu|$, as in this case $\lambda\equiv|\mu|$.
\end{rem}

\section{Necessary conditions of optimality}\label{S2} 
This section is devoted to introducing necessary optimality conditions for the following impulsive optimal control problem with time delays  (P):
\[
\text{Minimize} \ \ \ \Phi(x(0),x(T)) + \int_0^T l_0\left(t,\{x(t-h_k)\},\alpha(t)\right) dt + \int_{[0,T]} l_1(t, \alpha(t)) \mu(dt)
\]
over the set of control-trajectory triples $(\mu,\alpha,x)\in\U\times BV([-h,T],\R^n)$ satisfying 
\[
\begin{cases}
\ds x(t)=x(0)+\int_0^t f\left(s,\{x(s-h_k)\},\alpha(s)\right) ds + \int_{[0,t]}G(s, \alpha(s)) \mu(ds) \quad \text{$\forall t\in]0,T]$,} \\
x(t)=\zeta(t) \quad \text{a.e. }t\in [-h,0[,
\end{cases}
\]
satisfying the endpoint constraint
\[
  (x(0),x(T))\in\T,
\]     
and such that the function $t\mapsto l_1(t,\alpha(t))$ is $\mu$-integrable.

\vsm 

The functions $f$, $G$ and $\zeta$, as well as the control set $\U$, 
are as in the previous section.
The data now also comprise the cost functions $\Phi:\R^{2n}\to\R$, $l_0:[0,T]\times\R^{n(N+1)}\times\R^q\to\R$  and  $l_1:[0,T]\times\R^q\to \R^m$, and the target set   $\T\subset\R^{2n}$. 

\noindent Let us introduce the subset $\U_1\subset\U$ of controls, given by
\bel{U0}
\U_1:=\left\{(\mu,\alpha)\in\U: \ \ [0,T]\ni t \mapsto l_1(t,\alpha(t)) \ \ \text{is $\mu$-integrable}\right\}.
\eeq
We say that a triple $(\mu,\alpha,x)$ is a {\em feasible process} if it is an impulsive process as defined in Sec. \ref{S1},
  such that $(\mu,\alpha)\in\U_1$ and $(x(0),x(T))\in\T$. Thus, the optimization problem (P) can be reformulated as the minimization of the functional
\bel{J}
\J(\mu,\alpha,x) := \Phi(x(0),x(T)) + \int_0^T l_0\left(t,\{x(t-h_k)\},\alpha(t)\right) dt + \int_{[0,T]} l_1(t, \alpha(t)) \mu(dt)
\eeq
over feasible processes.  

\begin{definition} 
We say that a feasible process $(\bar\mu,\bar\alpha, \bar x)$ is {\em optimal} for (P) if
\[
\J(\bar\mu,\bar\alpha, \bar x) \leq \J(\mu,\alpha,x)
\]
for any feasible process $(\mu,\alpha,x)$.
\end{definition}
 
Given a feasible process $(\bar\mu,\bar\alpha,\bar x)$, which we will call the {\em reference process}, 
 in addition to hypotheses (H1)-(H4),   we shall invoke also the following assumptions.
\vsm

{\em  \begin{itemize} 

\item[{\bf (H5)}] The function $l_0$ satisfies hypotheses {\rm (H2)-(H3)} and the function $l_1$ satisfies assumption {\rm(H4)}, when they are inserted in place of $f$ and $G$, respectively.
\vsmm
\item[{\bf (H6)}]  $\Phi$ is Lipschitz continuous on a neighborhood of $(\bar x(0),\bar x(T))$ and $\T$ is a closed set. 
\end{itemize}
}
\noindent Furthermore,     for any $t\in \R$,  we define $\bar l_0[t]$ and $\bar f[t]$ as
\bel{barf}
\bar l_0[t]:= l_0(t,\{\bar x(t-h_k)\},\bar\alpha(t)), \qquad \bar f[t]:= f(t,\{\bar x(t-h_k)\},\bar\alpha(t)). 
\eeq 

Let us state the main result of this section. 
 \begin{theo}[Maximum Principle] \label{MP_Th}
Let $(\bar \mu,\bar \alpha,\bar x)$ be an optimal process for problem {\em (P)} and assume that the data satisfy {\rm (H1)}--{\rm (H4)} and {\rm (H5)}--{\rm (H6)}. Then, there   exist $\lambda \geq0$ and $p_k\in W^{1,1}([-h_k,T],\R^n)$, $k=0,\dots,N$,    such that
\bel{adj_k1}
\dot p_k(t)=0 \ \ \text{a.e.  $t\in[-h_k,0]$,} \quad p_k(t)=0 \ \ \forall t\in[(T-h_k)\vee0,T],
\eeq
for any $k=1,\dots,N$,
and satisfying conditions \eqref{nontriviality}--\eqref{maxham_Gg2} below:
\begin{gather}
\|p\|_{L^\infty} + \lambda \neq 0;\label{nontriviality} \\[1.2ex]
\begin{array}{l}
\big(-\dot p_0((t-h_0)\vee0),\dots, -\dot p_N((t-h_N)\vee0) \big) \\
\qquad\qquad\qquad\qquad \in {\rm co}\, \partial_{x_0,\dots,x_N} \big(p(t ) \cdot \bar f[t]- \lambda \bar l_0[t] \big) 
  \quad \text{a.e.  $t\in[0,T]$}; 
  \end{array} \label{adj_eqk_lagr} \\[1.2ex]
(p(0), -p(T)) \in N_{\T}(\bar x(0), \bar x(T)) +\lambda\,\partial \Phi(\bar x(0),\bar x(T));\label{trans_cond}  \\[1.2ex]
\begin{array}{l}
\ds\max_{a \in A(t)} \left\{ p(t)\cdot f(t,\{\bar x(t-h_k)\}, a) -\lambda l_0(t, \{\bar x(t-h_k) \}, a) \right\} \\ 
\ds\qquad\qquad\qquad\qquad\qquad\qquad\qquad =p(t) \cdot \bar f[t] -\lambda \bar l_0[t] \quad \text{a.e. $t\in[0,T]$;} 
\end{array}\label{maxham_f} \\[1.2ex]
\sup_{a \in A(t)} \left\{ \sigma_{_\K}(p(t) \cdot G(t, a)-\lambda l_1(t,a)) \right\} \leq 0 \ \ \ \forall t\in[0,T]; \label{maxham_Gg1} \\[1.2ex]
\sigma_{_\K}\big(p(t)\cdot G(t,\bar\alpha(t)) -\lambda l_1(t,\bar\alpha(t)) \big) = 0 \ \ \ \text{$\bar\mu$-a.e. $t\in[0,T]$,} \label{maxham_Gg2}
\end{gather}
where $p\in W^{1,1}([0,T],\R^n)$ is given by
\bel{def_p}
p(t) :=\sum_{k=0}^Np_k(t) \qquad \forall t\in[0,T].
\eeq
It follows from  \eqref{adj_eqk_lagr} and  \eqref{def_p}   that the adjoint arc $p$ satisfies
the following  ``advance” functional differential inclusion 
\bel{adj_eq_lagr}
-\dot p(t) \in \sum_{j=0}^N \,{\rm co}\,\tilde\partial_{x_j}\Big(p(t+h_j) \cdot \bar f[t+h_j] - \lambda \bar l_0[t+h_j] \Big)\, \chi_{[0,T]}(t+h_j) \ \ \ \text{a.e. $t\in[0,T]$}.
\eeq
(Here  $\tilde\partial_{x_j}$  denotes the projected limiting subdifferential onto the $j$-th delayed state coordinate).
Furthermore, when $\mu$ is scalar valued and $\K=\R_{\ge0}$, conditions \eqref{maxham_Gg1} and \eqref{maxham_Gg2} can be expressed as follows:
\begin{gather}
\sup_{a \in A(t)} \left\{ p(t) \cdot G(t, a)-\lambda l_1(t,a) \right\} \leq 0 \ \ \ \forall t\in[0,T], \label{maxham_Gg3} \\
p(t)\cdot G(t,\bar\alpha(t)) -\lambda l_1(t,\bar\alpha(t)) = 0 \ \ \ \text{$\bar\mu$-a.e. $t\in[0,T]$.} \label{maxham_Gg4}
\end{gather}
\end{theo}


\begin{rem}\label{Rem_pmp}  
Assumptions   (H1)--(H3)  are slightly different from  those used in the maximum principle with time delays in \cite[Thm. 2.1]{VB}, of which Thm. \ref{MP_Th} can be seen as an extension to our impulsive problem. In particular, condition (H2) is  stronger  than the usual local Lipschitz continuity condition in the state variable, while  (H3) is usually replaced by an integrably boundedness assumption in some $\varepsilon$-neighborhood of the reference trajectory $\{\bar x(t-h_k)\}$ in the $L^\infty$-norm.   Actually, the proof of Thm. \ref{MP_Th}  relies on  Ekeland's variational principle, but, because of the   presence of impulses, our approximating trajectories  converge  in general   to the reference trajectory only almost everywhere, not  in the $L^\infty$-norm (see Sec. \ref{MP_proof} below). Hypothesis (H3) is then used to guarantee the existence of solutions to \eqref{imp_del_sys} laying in some compact subset of $\R^n$ (as well as to ensure the continuity of the input-output map \eqref{input-output_map}). Anyway, as it is easy to see, with regard to  the drift term $f$, (H1)--(H3) imply the assumptions of  \cite[Thm. 2.1]{VB}.
%
\end{rem}

\begin{rem} In a nonsmooth setting,  condition \eqref{adj_eqk_lagr} 
is a more precise condition than \eqref{adj_eq_lagr}, as the subdifferential $\partial_{\{x_k\}}  (p\cdot f)$ is a subset, and in some cases a strict subset, of the product of projected partial subdifferentials $\tilde\partial_{x_0}(p\cdot f)\times\dots\times \tilde\partial_{x_N}(p\cdot f)$.\end{rem}

\begin{rem}
Conditions \eqref{maxham_Gg1}-\eqref{maxham_Gg2} locate the support of $\bar \mu$. This might appear to provide rather sparse information, but the fact that the optimal trajectory must satisfy the terminal constraints  and the conditions on the conventional control, embodies additional implicit information about $\bar\mu$. In fact, at least in the case without delays,  it is a simple matter to show that for linear convex problems, under a normality hypothesis, the conditions of Thm. \ref{MP_Th} are also sufficient for optimality of $(\bar x, \bar \alpha, \bar\mu)$ (see e.g. \cite{VP}). This shows the strength of the necessary conditions.
\end{rem}
In order to prove the Maximum Principle, we introduce the following notion. 

\begin{definition} 
Given a function $\Psi:\R^n\to\R^l$ and a subset $\C\subset \R^n$, we say that an impulsive process $(\bar\mu,\bar\alpha,\bar x)$ of  \eqref{imp_del_sys}  is a {\em $(\Psi,\C)$-boundary process} if $\bar x(0)\in \C$ and $\Psi(\bar x(T))$ belongs to the boundary of the {\em $(\Psi,\C)$-reachable set $\RR_\Psi^\C$ associated with \eqref{imp_del_sys}},  given by
\[
\RR_\Psi^\C := \{ \Psi(x(T)) \text{ : } (\mu,\alpha,x) \text{ is a process for \eqref{imp_del_sys} with } x(0)\in\C \}. 
\]
\end{definition}

We will first establish necessary conditions for $(\Psi,\C)$-boundary processes. These conditions are of interest in their own right. But they also can be used simply to derive the necessary conditions of   Thm. \ref{MP_Th}. 

\begin{theo}[Necessary conditions for boundary processes] \label{Boundary_Th}
Let $(\bar\mu,\bar\alpha,\bar x)$ be a $(\Psi,\C)$-boundary process for some $\Psi:\R^n\to\R^l$, which is Lipschitz continuous on a neighborhood of $\bar x(T)$, and for some closed subset $\C\subset\R^n$.  Let the data satisfy hypotheses  {\rm (H1)}--{\rm (H4)}. Then,  there exist $\eta\in\R^l\setminus\{0\}$ and $p_k\in W^{1,1}([-h_k,T],\R^n)$, $k=0,\dots,N$,  such that \eqref{adj_k1} is satisfied for any $k=1,\dots,N$
and conditions \eqref{adj_eqk}--\eqref{maxham_G2} below hold: 
\begin{gather}
\begin{array}{l}
\big(-\dot p_0((t-h_0)\vee0),\dots, -\dot p_N((t-h_N)\vee0) \big) \\
\qquad\qquad\qquad\qquad\qquad\in \,{\rm co}\, \partial_{x_0,\dots,x_N} \big(p(t ) \cdot \bar f[t] \big)  \quad \text{a.e. $t\in[0,T]$}; 
\end{array}\label{adj_eqk}\\[1.2ex]
p(0) \in N_{\C}(\bar x(0)); \label{trans1} \\[1.2ex]
-p(T)\in\eta \cdot \partial \Psi(\bar x(1));\label{trans2}\\[1.2ex]
p(t) \cdot \bar f[t] = \max_{a \in A(t)} \left\{ p(t)\cdot f(t,\{\bar x(t-h_k)\}, a)\right\} \ \ \ \text{a.e. $t\in[0,T]$}; \label{ham_f}
\end{gather}
\begin{gather}
\sup_{a \in A(t)} \left\{ \sigma_{_\K}(p(t) \cdot G(t, a)) \right\} \leq 0 \ \ \ \forall t\in[0,T];\label{maxham_G1} \\[1.2ex]
\sigma_{_\K}\big(p(t)\cdot G(t,\bar\alpha(t))\big) = 0 \ \ \ \text{$\bar\mu$-a.e. $t\in[0,T]$,} \label{maxham_G2}
\end{gather}
where $p$ is as in \eqref{def_p}.
It follows from \eqref{def_p}  and \eqref{adj_eqk} that $p$ satisfies
the following  ``advance” functional differential inclusion 
\bel{adj_eq}
-\dot p(t) \in \sum_{j=0}^N \,{\rm co}\,\tilde\partial_{x_j} \Big(p(t+h_j) \cdot \bar f[t+h_j] \Big)\, \chi_{[0,T]}(t+h_j) \quad \text{a.e. $t\in[0,T]$.}
\eeq
Furthermore, when $\mu$ is scalar valued and $\K=\R_{\ge0}$, conditions \eqref{maxham_G1} and \eqref{maxham_G2} can be strengthened as follows:
\begin{gather}
\sup_{a \in A(t)} \left\{ p(t) \cdot G(t, a) \right\} \leq 0 \ \ \ \forall t\in[0,T],\label{maxham_G3} \\
p(t)\cdot G(t,\bar\alpha(t)) = 0 \ \ \ \text{$\bar\mu$-a.e. $t\in[0,T]$.} \label{maxham_G4}
\end{gather}
\end{theo}

The proof of Thm. \ref{Boundary_Th} will be given in the next section.  We now show that, as anticipated,  Thm. \ref{MP_Th} can be deduced  as a corollary of Thm. \ref{Boundary_Th}.

\begin{proof}[Proof of Thm. \ref{MP_Th}]
Let $(\bar\mu,\bar\alpha, \bar x)$ be an optimal process for problem (P), as in the statement.  We introduce  a new control system, where we consider as  processes the 7-uples $(\mu,\alpha,y_1,\dots,y_5)$, in which $(\mu,\alpha)\in\U_1$, $y_1\in   BV([-h,T], \R^n)$, $y_2\in BV([0,T], \R)$, $(y_3,y_4,y_5)\in W^{1,1}([0,T],\R^n\times\R\times\R)$, and  
that satisfy 
\bel{tildeP}
\begin{cases}
\ds y_1(t) = y_1(0)+ \int_0^t f(s, \{ y_1(s-h_k)\}, \alpha(s))ds + \int_{[0,t]}G(s,\alpha(s)) \mu(ds) \quad \forall t\in]0,T], \\
\ds y_2(t) = y_2(0)+\int_0^t l_0(s, \{ y_1(s-h_k)\}, \alpha(s)) ds + \int_{[0,t]} l_1(s,\alpha(s)) \mu(ds) \quad \forall t\in]0,T], \\ 
\dot y_3(t) =0, \quad \dot y_4(t)=0, \quad \dot y_5(t)=0, \qquad \text{a.e. $t\in[0,T]$,}  \\
y_1(t)=\zeta(t) \qquad \text{a.e. $t\in[-h,0[$.}
\end{cases}
\eeq
Define the Lipschitz function $\Psi:\R^n\times\R\times\R^n\times\R\times\R\to\R^n\times\R\times\R$, by
\[
\Psi(z_1,z_2,z_3,z_4,z_5)=(z_1-z_3, z_2-z_4,z_5),
\]
the closed set $\C\subset \R^n\times\R\times\R^n\times\R\times\R$, by
\[
\C:=\{(z_1,z_2,z_3,z_4,z_5)  \text{ : } (z_1,z_3)\in\T, \ z_2=0, \ z_5\geq \Phi(z_1,z_3)+z_4 \},
\]
and let $\tilde\RR_{\Psi}^{\C}$ denote the $(\Psi,\C)$-reachable set associated with this new (impulsive) control system. Observe that the process $(\tilde\mu,\tilde\alpha,\tilde y_1,\dots,\tilde y_5)$, in which $(\tilde\mu,\tilde\alpha,\tilde y_1):=(\bar\mu,\bar\alpha, \bar x)$, and, for every $t\in[0,T]$,
\[
\begin{split}
&\tilde y_2(t):=\int_0^t l_0(s, \{ \bar x(s-h_k)\},  \bar\alpha(s)) ds + \int_{[0,t]}l_1(s,\bar\alpha(s))\bar\mu(ds), \\
&\tilde y_3(t)\equiv \bar x(T), \\
&\tilde y_4(t)\equiv\int_0^T l_0(t, \{ \bar x(t-h_k)\}, \bar\alpha(t)) dt + \int_{[0,T]}l_1(t,\bar\alpha(t))\bar\mu(dt) \ (=\tilde y_2(T)), \\
&\tilde y_5(t)\equiv\J(\bar\mu,\bar\alpha,\bar x). 
\end{split}
\]
satisfies \eqref{tildeP} with $(\tilde y_1,\tilde y_2,\tilde y_3,\tilde y_4,\tilde y_5)(0)\in\C$. Moreover,  $\Psi((\tilde y_1,\tilde y_2,\tilde y_3,\tilde y_4,\tilde y_5)(T))\in\partial \tilde\RR_{\Psi}^{\C}$, otherwise $(0,0, \J(\bar\mu,\bar\alpha,\bar x)-\rho)\in\tilde\RR_{\Psi}^{\C}$ for some $\rho>0$ sufficiently small, but this clearly contradicts the optimality of $(\bar\mu,\bar\alpha,\bar x)$.
 Thus, by Thm. \ref{Boundary_Th} there exist  
$\mathbf{p}_k:=(p_{1_k},p_{2_k},p_{3_k},p_{4_k},p_{5_k})\in W^{1,1}([-h_k,T],\R^{n+1+n+1+1})$ for $k=0,\dots,N$,  and $(\eta_1,\eta_2,\eta_3)\in \R^{n+1+1}\setminus \{0\}$ such that  
 \eqref{adj_k1} holds  with $\mathbf{p}_k$ replacing  $p_k$   for any  $k=1,\dots,N$,
 and such that  conditions (i)--(vii) below are met for some $\lambda\geq0$: 
 \[
\begin{split}
{\rm (i)} \ &\begin{array}{l}
(-\dot p_{1_0}(t), \dots, -\dot p_{1_N}((t-h_N)\vee 0)) \\
\qquad\qquad\qquad \in \,{\rm co}\,\partial_{x_0,\dots,x_N}\big\{p_1(t) \cdot \bar f[t] +p_2(t)\bar l_0[t] \big\}  \quad \text{a.e. $t\in[0,T]$};
\end{array}  \\[1.1ex]
{\rm (ii)}\  &\dot p_2(t)=0, \ \ \dot p_3(t)=0, \ \ \dot p_4(t)=0, \ \ \dot p_5(t)=0, \qquad \text{a.e. $t\in[0,T]$};  \\[1.1ex]
{\rm (iii)} \ &( p_1, p_3)(0)\in N_\T(\bar x(0),\bar x(T)) + \lambda \partial \Phi(\bar x(0),\bar x(T)), \ \  p_4(0)=\lambda, \ \ p_5(0)=-\lambda;\\[1.1ex]
{\rm (iv)} \ &(p_1,p_2,p_3,p_4,p_5)(T)\in -(\eta_1,\eta_2,\eta_3)\cdot\partial\Psi=(-\eta_1, -\eta_2, \eta_1,\eta_2,-\eta_3);\\[1.1ex]
{\rm (v)} \ &\begin{array}{l}
 \max_{a \in A(t)} \left\{ p_1(t)\cdot f(t,\{\bar x(t-h_k)\}, a) +p_2(t) l_0(t, \{\bar x(t-h_k) \}, a) \right\} \\
\qquad\qquad\qquad\qquad\qquad\qquad\qquad = p_1(t) \cdot\bar f[t] +p_2(t)\bar l_0[t] \quad \text{a.e. $t\in[0,T]$;} \end{array}\\[1.1ex]
{\rm (vi)} \ &\sup_{a \in A(t)} \left\{ \sigma_{_\K}(p_1(t) \cdot G(t, a)+p_2(t) l_1(t,a)) \right\} \leq 0 \ \ \ \forall t\in[0,T];  \\[1.1ex]
{\rm (vii)} \ &\sigma_{_\K}\big(p_1(t)\cdot G(t,\bar\alpha(t)) +p_2(t) l_1(t,\bar\alpha(t)) \big) = 0 \ \ \ \text{$\bar\mu$-a.e. $t\in[0,T]$},  
\end{split}
\]
where $\bar l_0$ and $\bar f$ are as in \eqref{barf} and $\mathbf{p}:=(p_{1},p_{2},p_{3 },p_{4},p_{5})$ is  given by 
$$
\mathbf{p}(t):=\sum_{k=0}^N\mathbf{p}_k(t) \qquad \forall t\in[0,T].
$$
In particular, (iii) follows by \cite[Ex. 1.11.26 and Ex. 2.9.11]{CLSW}.
Clearly,  in view of (i) and the very definition of $\mathbf{p}$,  for a.e. $t\in[0,T]$, $p_1$ satisfies \footnote{With $\tilde\partial_{x_j}$ we mean here the projected limiting subdifferential w.r.t. the $j$-th component of the variable $y_1$.}
\[
-\dot p_1(t) \in\sum_{j=0}^N\,{\rm co}\,\tilde\partial_{x_j}\big\{p_1(t+h_j) \cdot \bar f[t+h_j] +p_2(t+h_j)\bar l_0[t+h_j] \big\}\, \chi_{[0,T]}(t+h_j).  
\]
 By (ii) we deduce that $p_2$, $p_3$, $p_4$, and $p_5$ are constants and (iv) implies that $p_3=-p_1(T)$. Considering also (iii), we notice that $p_2=p_5=-p_4=-\lambda$ and $\eta_3=\eta_2$. We point out that $\| p_1\|_{L^\infty} + \lambda\neq 0$, otherwise $\eta_1=-p_1(T)=0$ and $\eta_2=\eta_3=p_4=-\lambda=0$, which contradicts $(\eta_1,\eta_2,\eta_3)\neq 0$. Therefore, $\lambda$ and the paths $p_k:=p_{1_k}$, $k=0,\dots,N$,  satisfy the requirements of Thm. \ref{MP_Th}.
\end{proof}

\section{Proof of Thm. \ref{Boundary_Th}} \label{MP_proof}
The proof consists of several steps. 
First of all,  we show that it is sufficient to prove the theorem for the auxiliary process associated with the given impulsive boundary process.  Then, we construct a sequence of optimization problems having as admissible controls only strict sense controls, and with costs measuring how much a trajectory is distant from the reference one in the $L^2$-norm. Using the Ekeland variational principle, minimizers are constructed for these problems that converge to the initial boundary process. Moreover, by applying the Maximum Principle in  \cite[Thm. 2.1]{VB} to these problems with reference to the above mentioned minimizers, we obtain in the limit a set of multipliers that meet conditions \eqref{adj_k1}--\eqref{maxham_Gg2}.
 
\vsm

{\em Step 1.}   In view of Prop. \ref{Th_equivalenza}, we can associate with the  $(\Psi,\C)$-boundary process $(\bar\mu,\bar\alpha,\bar x)$ the auxiliary process $(\bar\nu,\bar\alpha,\bar\w,\bar x)$  in which 
$\bar\w=\frac{d\bar\mu}{d|\bar\mu|}$ and  $\bar\nu$ is the scalar nonnegative measure defined as in \eqref{nu}, so that $(\bar\nu,\bar\alpha,\bar\w)\in\tilde\U$. In particular,  the trajectory $\bar x$ can be also  interpreted as a solution of the auxiliary control system,
\bel{aux_proof}
\begin{cases}
dx(t)=  f\left(t,\{x(t-h_k)\}, \alpha(t)\right)\,dt +\g(t, \alpha(t),\w(t)) \nu(dt),  \quad   t\in[0,T],  \\[1.2ex]
x(t) = \zeta(t) \ \ \ \text{a.e. }t\in [-h,0[,
\end{cases}
\eeq
where $\g$ is as in \eqref{tilde_g}.   Clearly,  $(\bar\nu,\bar\alpha,\bar\w,\bar x)$ turns out to be a $(\Psi,\C)$-boundary process for the $(\Psi,\C)$-reachable set associated with   control system \eqref{aux_proof}.

Suppose we have proved  Thm. \ref{Boundary_Th} for $(\bar\nu,\bar\alpha,\bar\w,\bar x)$. Then, since the drift term $f$ is the same for the original and the auxiliary control systems, there exist $\eta\in\R^l\setminus\{0\}$ and $p_k\in W^{1,1}([-h_k,T],\R^n)$, $k=0,\dots,N$,  such that
 \eqref{adj_k1} is satisfied for any $k=1,\dots,N$,
and such that conditions \eqref{adj_eqk}, \eqref{trans1}, \eqref{trans2}, \eqref{ham_f}, and \eqref{adj_eq} are met. Moreover, conditions \eqref{maxham_G3} and \eqref{maxham_G4} for the auxiliary control system reduce respectively to
\bel{provv1}
\sup_{a\in A(t), \ w\in\tilde\K}\left\{ \frac{p(t)\cdot G(t,a)w}{1+\sum_{i=1}^m |\sum_{j=1}^m g_{ij}(t,a) w^j|} \right\} \leq 0 \qquad \forall t\in[0,T]
\eeq
and
\bel{provv2}
\frac{p(t)\cdot G(t,\bar\alpha(t)) \bar\w(t)}{1+ \sum_{i=1}^m |\sum_{j=1}^m g_{ij}(t,\bar\alpha(t)) \bar\w^j(t)|}=0 \qquad \text{$\bar\nu$-a.e. $t\in[0,T]$.}
\eeq
Let us observe that for any $w\in\K$ there exist $\tilde w\in\tilde\K$ and $r\ge0$ such that $w=r\tilde w$. Moreover, from \eqref{nu} it follows that for any $B\in\BB$ one has $\bar\nu(B)>0$ provided $\bar\mu(B)>0$. Hence, it is immediate to see that  \eqref{provv1} and  \eqref{provv2} imply  \eqref{maxham_G1} and \eqref{maxham_G2} for the original process $(\bar\mu,\bar\alpha,\bar x)$, respectively. 
Furthermore, in case $\K=\R_{\ge0}$ then $\bar\mu\equiv |\bar\mu|$, so that $\bar\w\equiv 1$ in view of \eqref{nu}.  Accordingly, from conditions  \eqref{provv1} and  \eqref{provv2} we can easily deduce the strengthened relations \eqref{maxham_G3} and \eqref{maxham_G4}, respectively.
 Therefore, Thm. \ref{Boundary_Th} is proved for  $(\bar\mu,\bar\alpha,\bar x)$  as soon as it is proved for the auxiliary process  $(\bar\nu,\bar\alpha,\bar\w,\bar x)$, and this will be our goal from now on.
\vsm

{\em Step 2.} Let $(\xi_i)_i\subset\R^l\setminus\RR_\Psi^\C$ be a sequence such that $\xi_i\to \Psi(\bar x(T))$, which exists as $\Psi(\bar x(T))\in\partial\RR_\Psi^\C$. 
From Lemma \ref{lemma_dense} it follows that  there exist sequences $(\tilde m_i)_i\subset L^1([0,T],\R_{\geq 0})$ and $(\tilde\alpha_i,\tilde\w_i)_i\subset \A\times\W$ such that 
\bel{convergenze}
\tilde m_i(t)dt \weak \bar\nu(dt),\qquad \ell(\{t\in[0,T] \text{ : } (\tilde\alpha_i,\tilde\w_i)(t)\neq(\bar\alpha,\bar\w)(t)\})\to0, 
\eeq
and $\g(t,\tilde\alpha_i(t),\tilde\w_i(t))\tilde m_i(t)dt \weak \g(t,\bar\alpha(t),\bar\w(t))\bar\nu(dt)$. Hence,  by Lemma \ref{lemma_misure}, (ii)-(iii),  possibly passing to a subsequence (we do not relabel here and in the rest of the proof), we have that there exists a subset $\E\subset[0,T]$ with $T\in\E$  and $[0,T]\setminus\E$ at most countable, such that  
$$
\lim_i\int_{[0,t]}\g(t,\tilde\alpha_i(t),\tilde\w_i(t))\tilde m_i(t)\,dt=\int_{[0,t]}\g(t,\bar\alpha(t),\bar\w(t))\bar\nu(dt) \quad \forall t\in\E.
$$
Since the sequence $\big(\g(t,\tilde\alpha_i(t),\tilde\w_i(t))\tilde m_i(t)dt\big)_i$ is uniformly bounded in total variation,  by (H4) and \eqref{convergenze}, and since (H3) implies $|f(t,\{\bar x(t-h_k)\},\tilde\alpha_i(t))|\leq c(t) (1+(N+1)\|\bar x \|_{L^\infty}) \in L^1([0,T],\R_{\geq0})$,
from  Lemma \ref{lemma_convergenza} it now follows that there exists a sequence $(\tilde x_i)_i$ of functions from $[-h,T]$ to $\R^n$ whose restriction to $[0,T]$ is absolutely continuous,  such that, for any $i$, $\tilde x_i$ is a solution to
$$
\begin{cases}
 x(t) = \bar x(0) + \int_0^t f(s,\{ x(s-h_k)\},\tilde\alpha_i(s)) ds \\
\qquad\qquad\qquad\qquad\qquad + \int_{[0,t]} \g(s,\tilde\alpha_i(s),\tilde\w_i(s)) \tilde m_i(s)ds \ \ \   \forall t\in]0,T] \\[1.5ex]
x(t) =\zeta(t) \ \ \text{ a.e. $t\in[-h,0[$}, \ \ \ x(0)=\bar x(0).
\end{cases}
$$
Moreover, $d\tilde x_i\weak d\bar x$, so that $(d\tilde x_i)_i$ is uniformly bounded in total variation and $(\tilde x_i)_i$ is uniformly bounded in $L^\infty$ by \eqref{def_integrale}. 
Furthermore, $\tilde x_i(t)\to \bar x(t)$ for all $t\in\E$.   Hence, by the dominated convergence theorem we deduce that $\tilde x_i\to \bar x$ in $L^2([0,T],\R^n)$, so that $\eps_i\to0$, being $(\eps_i)_i\subset \R_{\geq0}$ the sequence defined by
\[
\eps_i^2:= \int_0^T |\tilde x_i(t)-\bar x(t)|^2 dt + |\tilde x_i(T) - \bar x(T)|^2 + |\xi_i- \Psi(\tilde x_i(T))|.
\]
Possibly passing to a subsequence, we can suppose that $\eps_i\le1$ for all $i$.  Let $(C_i)_i$ be the sequence of $L^1$ functions defined by
\bel{C_i}
C_i(t):= i + \max_{j\leq i} \{\tilde m_j(t) \}. 
\eeq
Clearly, this sequence is  monotone, $\tilde m_i(t)\leq C_i(t)$,  and $C_i(t)\to+\infty$ for all $t\in[0,T]$. 
\vsmm 
For each $i$, consider the following (non-impulsive) optimization problem with time delays,
\begin{equation*} 
{\rm (\tilde P_i)}
\left\{
\begin{array}{l}
\ \ \ \text{Minimize} \,\,\,\,    \int_0^T |x(t)-\bar x(t)|^2 dt + |x(T) - \bar x(T)|^2 + |\xi_i- \Psi(x(T))| \\[1.0ex] 
\text{over controls 
$(m,\alpha,\w)\in L^1([0,T],\R_{\geq0}) \times \M([0,T],\R^q\times\R^m)$} \\[1.2ex]
\text{and arcs $x\in L^\infty([-h,T],\R^n)\cap W^{1,1}([0,T],\R^n)$, satisfying} 
\\[1.0ex]
\dot x(t)= f(t,\{x(t-h_k)\},\alpha(t)) + \g(t,\alpha(t),\w(t)) m(t) \ \ \text{a.e. $t\in[0,T]$} \\[1.0ex]
x(t)=\zeta(t) \quad \text{a.e. }t\in[-h,0[, \quad
x(0)\in\C, \\[1.0ex]
(m(t),\alpha(t),\w(t))\in [0,C_i(t)]\times A(t)\times \tilde \K \ \ \text{a.e. $t\in[0,T]$}.
\end{array}
\right.
\end{equation*}

  For each $i$, let $\Gamma_i$ be the set of elements  $(m,\alpha,\w,x_0)$ for which $x_0\in\C$ and  $(m,\alpha,\w)$ is a measurable control such that $(m(t),\alpha(t),\w(t))\in [0,C_i(t)]\times A(t)\times \tilde \K$ a.e.. Under our hypotheses,  for every  $(m,\alpha,\w,x_0)\in\Gamma_i$ there exists exactly one solution 
$$
x:=x[m,\alpha,\w,x_0]\in L^\infty([-h,T],\R^n)\cap W^{1,1}([0,T],\R^n)
$$ 
to the delayed control system in $(\tilde P_i)$  with initial condition $x(0)=x_0$,  so that problem $ (\tilde P_i)$ can be reformulated as 
\[
\begin{cases}
\text{Minimize} \ \J_i(m,\alpha,\w,x_0):= \int_0^T |x(t)-\bar x(t)|^2 dt + |x(T) - \bar x(T)|^2 + |\xi_i- \Psi(x(T))| \\
\text{over $(m,\alpha,\w,x_0)\in\Gamma_i$.}
\end{cases}
\]
 The set $\Gamma_i$  is a complete metric space, when equipped with the metric  $\d$ defined by
\[
\begin{split}
\d((m,\alpha,\w,x_0),(\tilde m, \tilde\alpha,\tilde\w, \tilde x_0))&:=  \|m-\tilde m\|_{L^1} + |x_0- \tilde x_0| \\
&\qquad\qquad +\ell(\{ t \in [0,T] \text { : } (\alpha,\w)(t)\neq (\tilde\alpha,\tilde\w)(t) \}) .
\end{split}
\]
(see \cite[Lemma 1, p. 202]{Cl}). Moreover, by the continuity of the input-output map associated with the (conventional) delayed control system in  $(\tilde P_i)$,  for each $i$ there exists a  function $\rho_i:\R_{\ge0}\to\R_{\ge0}$ with $\lim_{d\to0^+}\rho_i(d)=0$, such that, for any pair $(m,\alpha,\w,x_0)$, $(\tilde m, \tilde\alpha,\tilde\w, \tilde x_0)\in\Gamma_i$,
$$
\|x[m,\alpha,\w,x_0]-x[\tilde m, \tilde\alpha,\tilde\w, \tilde x_0]\|_{L^\infty}\le\rho_i\big(\d((m,\alpha,\w,x_0),(\tilde m, \tilde\alpha,\tilde\w, \tilde x_0))\big).
$$
As a consequence, the map $\J_i$ is continuous on $\Gamma_i$ with respect to $\d$.

\vsm
{\it Step 3.} By the previous arguments,  $(\tilde m_i,\tilde\alpha_i,\tilde\w_i,\tilde x_i(0))$ is an $\eps_i^2$-minimizer of   problem ${\rm (\tilde P}_i)$,  thus  Ekeland's variational principle yields the existence of a sequence $(m_i,\alpha_i,\w_i,x_{0_i})$ which is optimal for the following  optimization problem  
\[
{\rm (P}_i)
\begin{cases}
\ \ \ \text{Minimize} \ \  \J_i(m,\alpha,\w,x_0) \\
\qquad\qquad\quad\qquad+\eps_i\left( |x_0- x_{0_i}| + \int_0^T [|m(t)-m_i(t)| + \vartheta_i(t,\alpha(t),\w(t))] dt \right) \\[1.5ex]
\text{over } \ (m,\alpha,\w,x_0)\in\Gamma_i,
\end{cases}
\]
where $\vartheta_i:[0,T]\times\R^q\times\R^m\to\{0,1\}$ is defined by
\[
\vartheta_i(t,a,w):=
\begin{cases}
0 \qquad \text{if } (a,w)=(\alpha_i(t),\w_i(t)), \\
1 \qquad \text{otherwise.}
\end{cases}
\]
Moreover, it also holds 
\bel{ek_dist}
\d((m_i,\alpha_i,\w_i,x_{0_i}), (\tilde m_i,\tilde\alpha_i,\tilde \w_i,\tilde x_i(0))) \leq \eps_i.
\eeq
By \eqref{convergenze}, \eqref{ek_dist}, and the fact that $\tilde x_i(0)\equiv\bar x(0)$ we get
\begin{gather}
\ell(\{t\in[0,T] \text{ : } (\alpha_i(t),\w_i(t))\neq (\bar\alpha(t),\bar \w(t))\}) \to 0,\label{alfa_ek} \\
m_i(t)dt \weak \bar\nu(dt), \label{mu_ek} \\
x_{0_i}\to\bar x(0). \label{x0_ek}
\end{gather}
By \eqref{ek_dist},  the fact that   $(\tilde m_i,\tilde\alpha_i,\tilde\w_i,\tilde x_{0_i})$ is an $\eps_i^2$-minimizer for problem ${\rm (\tilde P}_i)$, and the fact that   $(m_i,\alpha_i,\w_i,x_{0_i})$  is optimal for ${\rm (P}_i)$,  we deduce that $\J_i(m_i,\alpha_i,x_{0_i})\to0$, so that the trajectories  $x_i:=x[m_i,\alpha_i,\w_i,x_{0_i}]$ satisfy $x_i\to \bar x$ in $L^2([0,T],\R^n)$ and $x_i(T)\to\bar x(T)$. Hence, possibly up to a subsequence,  we have
\bel{conv_ae}
x_i\to \bar x\ \  \text{ in $L^2$}, \qquad x_i(t)\to \bar x(t) \ \ \text{ a.e. }t\in[0,T], \qquad x_i(T)\to\bar x(T).
\eeq
By (H4), \eqref{mu_ek}, and \eqref{x0_ek} there exists $\tilde C>0$ such that $(x_{0_i})_i\subset \tilde C\B_n$ and the sequence $\big(\g(t,\alpha_i(t),\w_i(t)) m_i(t) dt \big)_i$ is bounded in total variation by $\tilde C$. Accordingly, using (H3) we get
\[
\begin{split}
|x_i(t)| &\leq |x_{0_i}| + \int_{0}^t |f(s,\{x_i(s-h_k)\},\alpha_i(s))| \, ds + \int_{0}^t |\g(s,\alpha_i(s),\w_i(s))| \, |m_i(s)| ds \\
& \leq 2\tilde C +  \int_{0}^t c(s) (1+ |\{x_i(s-h_k)\}|) \, ds,
\end{split}
\]
from which we immediately deduce
\[
\|x_i\|_{L^{\infty}(0,t)}\leq 2\tilde C +  \int_{0}^t c(s) \Big(1+(N+1)\big( \|x_i\|_{L^{\infty}(0,s)} + \|\zeta\|_{L^{\infty}(-h,0)} \big)\Big) \, ds.
\]
A straightforward application of the Gronwall's Lemma implies that $(x_i)_i$ is a uniformly bounded sequence in $L^\infty$, so that there exists $M>0$ such that
\bel{bound_xi}
\|\bar x\|_{L^\infty([-h,T])}\leq M,\qquad\qquad \|x_i\|_{L^\infty([-h,T])}\leq M\ \ \ \text{for any $i$.}
\eeq
As a consequence, we obtain
\[
\begin{split}
\int_{0}^T |\dot x_i(t)| dt &\leq \int_{0}^T |f(t,\{x_i(t-h_k)\},\alpha_i(t))| \, ds + \int_{0}^T |\g(t,\alpha_i(t),\w_i(t))| \, |m_i(t)| dt \\
& \leq \int_{0}^T c(t) \Big(1+ (N+1) \big( M + \|\zeta\|_{L^{\infty}(-h,0)} \big) \Big) \, ds + \tilde C. \\
\end{split}
\]
Therefore, the sequence of measures $(dx_i)_i$ associated with the $x_i$'s is uniformly bounded in total variation and, in view of \eqref{x0_ek}-\eqref{conv_ae}, satisfies
$$
\int_{[0,t]}dx_i(s)\to \int_{[0,t]}d\bar x(s) \quad\forall t\in\tilde\E,
$$
where $\tilde\E\subset[0,T]$ has full Lebesgue measure and contains $T$. 
Hence, Lemma \ref{lemma_misure}, (i), yields that
$dx_i\weak d\bar x$.  Thanks to  \eqref{alfa_ek}, \eqref{conv_ae}, hypothesis (H3) and \eqref{bound_xi}, we can apply the dominated convergence theorem to deduce that
\bel{conv_f}
f(t,\{x_i(t-h_k)\},\alpha_i(t))dt \weak f(t,\{\bar x(t-h_k)\},\bar\alpha(t))dt.
\eeq
Since $\g(t,\alpha_i(t),\w_i(t))m_i(t)dt=dx_i(t)-f(t,\{x_i(t-h_k)\},\alpha_i(t))dt$, we obtain that
\bel{conv_G}
\g(t,\alpha_i(t),\w_i(t))m_i(t)dt \weak \g(t,\bar\alpha(t),\bar\w(t))\bar\nu(dt).
\eeq

\vsm
{\it Step 4.} Problem ${\rm (P}_i)$ is a non-impulsive optimization problem with time delays in the dynamics,  for which  a Maximum Principle is available. In particular, by applying \cite[Thm. 2.1]{VB} with reference to the minimizer $(m_i,\alpha_i,\w_i,x_i)$ we deduce the existence of arcs   $p_{k_i}\in W^{1,1}([-h_k,T],\R^n)$, $k=0,\dots, N$, such that
\bel{def_pki}
\dot p_{k_i}(t)=0 \ \ \text{for a.e. $t\in[-h_k,0]$,} \qquad p_{k_i}(t)=0 \ \ \forall t\in[(T-h_k)\vee0,T],
\eeq
for $k=1,\dots,N$, and satisfying  the following conditions: \footnote{Since problem (P$_i$)  has free terminal point,  the cost multiplier can be taken equal to 1.}
\begin{gather}
\begin{array}{l}
(-\dot p_{0_i}((t-h_0)\vee0),\dots,-\dot p_{N_i}((t-h_N)\vee0)) \\
\qquad\qquad\in {\rm co}\,  \partial_{x_0,\dots,x_N}\big(p_i(t) \cdot  f_i[t]\big) -2 (x_i(t) - \bar x(t)) \quad \text{ a.e. $t\in[0,T]$}; 
\end{array}\label{adj_ek} \\[1.2ex]
p_i(0) \in N_\C(x_i(0)) + \eps_i \B_n; \label{trans1_ek} \\[1.2ex]
-p_i(T) \in \partial_x |\xi_i - \Psi(x_i(T))| +2(x_i(T)-\bar x(T)); \label{trans2_ek}\\[1.2ex]
\begin{array}{l}
p_i(t) \cdot \big(f_i[t] + \g(t,\alpha_i(t),\w_i(t)) m_i(t)\big) \\
 \qquad\ds=\max_{ a\in A(t), \ w\in\tilde\K, \ m\in[0,C_i(t)]}\big\{ p_i(t) \cdot \big[ f(t,\{x_i(t-h_k)\}, a) +\g(t,a,w)m \big]
 \\
\ds\quad\qquad\qquad\qquad
- \eps_i(\vartheta_i(t,a,w) + |m-m_i(t)|)\big\} \quad \text{a.e. $t\in[0,T]$,}
\end{array} \label{maxham_ek}
\end{gather}
where, for any $t\in[0,T]$, $p_i(t)$ and $f_i[t]$ are given by
\bel{fpi}
p_i(t):=\sum_{k=0}^N p_{k_i}(t), \qquad\qquad f_i[t]:= f(t,\{ x_i(t-h_k)\},\alpha_i(t)).
\eeq
We now deduce conditions \eqref{adj_eqk}--\eqref{maxham_G2} by passing to the limit in relations \eqref{adj_ek}--\eqref{maxham_ek}. 
To begin with, observe that, from \eqref{adj_ek} and the very definition of $p_i$ in \eqref{fpi}, using  \eqref{bound_xi} together with (H2), we get
\[
\begin{array}{l}
|p_i(t)|  \leq |p_i(T)| + \int_{[t,T]} \sum_{k=0}^N |\dot p_{k_i}(s)| \,ds \\
\ \ \quad\quad\leq |p_i(T)|+ \sum_{k=0}^N \int_{[t,T]}   |p_i(s+h_k)| L_M(s+h_k) \chi_{[0,T]}(s+h_k) \, ds \\
\qquad\qquad\qquad\qquad\qquad\qquad\qquad\qquad\qquad\qquad\qquad+ 2\int_{[0,T]} |x_i(s) - \bar x(s)|\, ds.
\end{array}
\]
Since $(m_i,\alpha_i,\w_i,x_i)$ is a strict sense process and $\xi_i\in\R^l\setminus\RR_\Psi^\C$ then $|\xi_i- \Psi(x_i(T))|\neq0$, hence \eqref{trans2_ek} and the Jacobian chain rule imply that there exists $\eta_i\in\partial\B_l$ such that
\bel{trans_i}
-p_i(T)\in\eta_i\cdot\partial\Psi(x_i(T)) + 2 (x_i(T)-\bar x(T)).
\eeq
Thus, by \eqref{conv_ae} and the fact that $\Psi$ is Lipschitz continuous in a neighborhood of $\bar x(T)$ it follows that $(p_i(T))_i\subset\R^n$ and  $\big(\int_{[0,T]} |x_i(s) - \bar x(s)|\, ds\big)_i\subset \R$ are uniformly bounded sequences. 
Thereby, a standard application of the Gronwall's Lemma to the map $t\mapsto \|p_i\|_{L^\infty([t,T])}$ allows us to deduce that $(p_i)_i\subset W^{1,1}([0,T])$ is a uniformly bounded sequence in $L^{\infty}$, namely, there exists $P>0$ such that
\bel{bound_p}
\|p_i\|_{L^\infty} \leq  P     \qquad \forall i\in\N.
\eeq 
Accordingly, $(p_{k_i})_i$ is a sequence with uniformly integrably bounded derivatives for any $k=0,\dots,N$,  such that $p_{k_i}(T)=0$ for $k=1,\dots,N$ in view of \eqref{def_pki}, and  the sequence $(p_{0_i}(T))_i\equiv (p_i(T))_i$ is uniformly bounded by the previous arguments. In view of the Ascoli-Arzel\'a's Theorem there exist functions $\tilde p_k\in W^{1,1}([-h_k,T],\R^n)$, $k=0,\dots,N$, and a subsequence of $(p_{k_i})_i$ such that
$
p_{k_i}\to \tilde p_k$ in $L^\infty$. 
As a consequence, for $p_i$ as in \eqref{fpi}, one has
\bel{tildepi}
p_i \to \tilde p := \sum_{k=0}^N \tilde p_k \quad \text{in $L^\infty$}. 
\eeq
Furthermore, the upper semicontinuity of the Clarke generalized Jacobian 
and \eqref{conv_ae} yield the existence of a sequence $(\tilde r_i)_i$ of measurable functions from $[0,T]$ into $\R_{\geq0}$, such that $\tilde r_i(t)\to0$ for a.e. $t\in[0,T]$, and for which we have
\[
D_{x_0,\dots,x_N} f(t, \{x_i(t-h_k)\}, \bar\alpha(t)) \subset D_{x_0,\dots,x_N} \bar f[t] + \tilde r_i(t) \B,
\]
where $\bar f[\cdot]$ is as in \eqref{barf}. Observe that, in view of (H2), the sequence $(\tilde r_i)_i$ is uniformly integrably bounded. In particular, we have that $\tilde r_i(t)\le 2L_M(t)$ a.e, where $M$ is as in \eqref{bound_xi}.
Therefore, from \eqref{adj_ek} and \eqref{bound_p},  setting
\[
 \Omega_i:=\{t\in[0,T] \text{ : } \alpha_i(t)=\bar\alpha(t) \},  
 \]
 for any $t\in\Omega_i$,  we get \footnote{We recall that 
$
p\cdot DG (x)= {\rm co}\,\partial(p\cdot G)(x)$ for any $p\in\R^l$ and for all $x\in\R^k$, 
see \cite[Def. 6.2.2]{OptV}.}
\[
\begin{split}
&\big(-\dot p_{0_i}((t-h_0)\vee0),\dots, -\dot p_{N_i}((t-h_N)\vee0) \big)\\
&\quad\quad\in p_i(t) \cdot D_{x_0,\dots,x_N} \bar f[t] + \Big(|p_i(t)| \tilde r_i(t) +2 |x_i(t) - \bar x(t)|\Big) \B \\
&\quad\quad \subset {\rm co}\, \partial_{x_0,\dots,x_N} \big(\tilde p(t)\cdot  \bar f[t]\big) + \Big( |p_i(t) - \tilde p(t)| L_M(t) + P \tilde r_i(t) +2 |x_i(t) - \bar x(t)| \Big) \B.
\end{split}
\]
In particular, we have shown that 
\[
\big(-\dot p_{0_i}((t-h_0)\vee0),\dots, -\dot p_{N_i}((t-h_N)\vee0)) \big)\in \,{\rm co}\, \partial_{x_0, \dots,x_N} \Big( \tilde p(t) \cdot  \bar f[t]\Big) + r_i(t)  
\] 
for a.e. $t\in \Omega_i$,
where $r_i:[0,T]\to\R_{\geq0}$ is given by
$$
r_i(t):=P\, \tilde r_i(t) +2  |x_i(t) - \bar x(t)| + \|p_i - \tilde p \|_{L^{\infty}} L_M(t).
$$
By \eqref{conv_ae},  \eqref{tildepi},  and the  boundedness property of  $\tilde r_i$, it follows that  the sequence $r_i$ converges to 0  a.e. on $[0,T]$ and is uniformly integrably bounded, hence
  the dominated convergence theorem  implies  $r_i\to 0$ in $L^1$.
Therefore, from the compactness of trajectories theorem 
(see \cite[Thm. 2.5.3]{OptV}) it follows that there exist functions $p_k\in W^{1,1}([-h_k,T],\R^n)$, $k=0,\dots,N$, and a subsequence of $(p_{k_i})_i$ such that
\bel{conv_pi}
p_{k_i}\to p_k \ \ \text{in $L^\infty$,}
\eeq
and  
\[
\big(-\dot p_{0}((t-h_0)\vee0),\dots, -\dot p_{N}((t-h_N)\vee0) \big)\in {\rm co} \, \partial_{x_0, \dots,x_N} \Big( \tilde p(t) \cdot  \bar f[t]\Big) \quad \text{a.e. $t\in[0,T]$}. 
\]
By the uniqueness of the uniform limit we deduce $p_k \equiv \tilde p_k$ for every $k=0,\dots,N$, so that $p\equiv \tilde p$ and the adjoint equation \eqref{adj_eqk} is confirmed. Moreover, in view of \eqref{def_pki}, also \eqref{adj_k1} holds for $k=1,\dots,N$. 

Passing to the limit of a proper subsequence in \eqref{trans_i}, using \eqref{conv_ae},  the upper semicontinuity of the Clarke's generalized Jacobian,  and the fact that $(\eta_i)_i$ is a bounded sequence, we deduce that there exists $\eta\in\partial\B_l$ (hence, $\eta\neq0$) for which the transversality condition \eqref{trans2} at the final point holds. Furthermore, from \eqref{trans1_ek}, the properties of the limiting normal cone, and \eqref{x0_ek} we easily deduce the  transversality condition  \eqref{trans1} at the initial point.

\vsm

Now we prove the  maximality conditions  \eqref{ham_f} and \eqref{maxham_G3}, where the last one,  for the auxiliary optimization problem,  reduces to
\bel{maxham_G3_aux}
\sup_{a\in A(t), \, w\in\tilde\K} \{p(t) \cdot \g(t,a,w)\} \leq 0 \qquad \forall t\in[0,T].
\eeq
In the following, we will use the fact that, given a sequence $(\NN_i)_i$ of subsets of $[0,T]$ with $\ell(\NN_i)\to0$, then there is a subsequence (we do not relabel) such that \footnote{Indeed, it is enough to consider a subsequence such that $\sum_i \ell(\NN_i) < \infty$. Then,    $\ell(\{t\in[0,T]\setminus\NN_j  \ \text{ for $j\geq i$}\}) \geq \ell(\{t\in[0,T]\setminus \cup_{j\geq i} \NN_j \}) \geq T - \sum_{j\geq i} \ell(\NN_j) \to T$ as $i\to\infty$.} 
\[
\lim_i \ \ell(\{t \text{ : } t\in[0,T]\setminus\NN_j  \ \text{ for any $j\geq i$}\}) = T. 
\]
By \eqref{mu_ek} we have that $(m_i)_i$ is uniformly bounded in $L^1$, hence 
\[
\ell(\{t\in[0,T] \text{ : } m_i(t) > 1/\sqrt{\eps_i} \})\to 0 \qquad \text{as $i\to\infty$.}
\]
Moreover, using \eqref{bound_p}, \eqref{bound_f}, and \eqref{bound_xi}, we get that the sequence   $(\varphi_i)_i$ given by 
\[
\varphi_i(t):=\sup_{a\in A(t)} |p_i(t) \cdot f(t, \{x_i(t-h_k)\}, a)|, \quad t\in[0,T],
\]
is uniformly pointwise bounded by $P c(t) \big(1+ (N+1)(M+\|\zeta\|_{L^\infty(-h,0)})\big)$, so that 
\[
\ell(\{t\in[0,T] \text{ : } \varphi_i(t) > \sqrt{C_i(t)} \})\to 0 \qquad \text{as $i\to\infty$,}
\]
since $C_i(t)\to +\infty$ as $i\to\infty$ for any $t\in[0,T]$. Furthermore, since \eqref{maxham_ek} implies that $m_i(t)=C_i(t)$ as soon as $p_i(t)\cdot \g(t,\alpha_i(t),\w_i(t))>\eps_i$, again from \eqref{mu_ek} it holds
\[
\ell(\{t\in[0,T] \text{ : } p_i(t)\cdot \g(t,\alpha_i(t),\w_i(t))>\eps_i \})\to 0 \qquad\text{as $i\to\infty$.}
\]
Let $\S_i$ bet is the subset of $[0,T]$ of points that satisfy the following conditions: 
\begin{gather}
m_j(t) \leq \frac{1}{\sqrt{\eps_i}} \qquad \forall j\geq i \label{bound_mi} \\[1.2ex]
\sup_{a\in A(t)} |p_j(t)\cdot f(t,\{x_j(t-h_k)\},a)| \leq \sqrt{C_j(t)}\qquad \forall j\geq i \label{bound_pfi}\\[1.2ex]
p_j(t)\cdot \g(t,\alpha_j(t),\w_j(t)) \leq \eps_j \qquad \forall j\geq i \label{bound_pGi} \\[1.2ex]
\alpha_j(t) =\bar\alpha(t) \qquad \forall j\geq i \label{alfa_i} \\[1.2ex]
\begin{array}{l}
\ds \max_{\substack{a\in A(t), \, w\in\tilde\K \\ m\in[0,C_j(t)]}}\big\{ p_j(t) \cdot \big(f(t,\{x_j(t-h_k)\}, a) + \g(t,a,w)m \big) - \eps_j(\vartheta_j(t,a,w) \\
\qquad+ |m-m_j(t)|)\big\} =p_j(t) \cdot\big( f_j[t] + \g(t,\alpha_j(t),\w_j(t))m_j(t) \big) \quad \forall j\geq i
\end{array} \label{ham_i}  \\[1.2ex]
x_j(t-h_k) \to \bar x(t-h_k) \qquad \forall k=0,\dots,N, \ \ \text{as }j\to\infty. \label{ae_xi}
\end{gather}
By the previous arguments, \eqref{alfa_ek}, \eqref{conv_ae}, and \eqref{maxham_ek}, it turns out that, up to a subsequence, $\ell(\S_i)\to T$. Clearly, $\S:=\cup_i \S_i$ is a dense subset of $[0,T]$.

  Given $t\in\S$,   let $i\in\N$  be such that $t\in\S_i$. Fix arbitrary   $\bar w\in\tilde\K$ and $\bar a\in A(t)$.  Using \eqref{bound_mi}, \eqref{bound_pGi}, \eqref{alfa_i}  and taking $a=\bar a$ and $m=0$, from \eqref{ham_i} we obtain
\[ 
p_j(t) \cdot f(t,\{x_j(t-h_k)\},\bar\alpha(t)) + \sqrt{\eps_j} \geq p_j(t) \cdot f(t,\{x_j(t-h_k)\},\bar a) -\eps_j \Big(1+ \frac{1}{\sqrt{\eps_j}}\Big) 
\]
for any $j\geq i$. 
By \eqref{conv_pi} and \eqref{ae_xi} we can let $j\to\infty$ in the above relation. Therefore, we deduce that the drift-maximality condition \eqref{ham_f} holds in the full measure subset $\S$.

\noindent Now, for any $j\geq i$,  choose $\bar w_j\in\tilde\K$,  $\bar a_j \in A(t)$ satisfying
\bel{supremum}
p_j(t) \cdot \g(t,\bar a_j,\bar w_j) C_j(t) \geq \sup_{a \in A(t), \, w\in\tilde\K} \{p_j(t) \cdot \g(t,a,w) C_j(t) \} - \eps_j.
\eeq
By \eqref{bound_mi}--\eqref{bound_pGi}, \eqref{supremum}, and taking $a=\bar a_j$ and $m=C_j(t)$, from \eqref{ham_i} we obtain
\[
\sqrt{C_j(t)} + \sqrt{\eps_j} \geq -\sqrt{C_j(t)} + \sup_{a\in A(t), \ w\in\tilde\K}\{p_j(t) \g(t,a,w)  \} C_j(t)-\eps_j -\eps_j-\eps_j C_j(t)
\]
for any $j\geq i$. 
Since $C_j(t)\to \infty$ as $j\to\infty$, if we divide the above inequality for $C_j(t)$ and we let $j\to\infty$, by the continuity of $p\mapsto \sup_{a\in A(t), \ w\in\tilde\K} \{ p \cdot \g(t,a,w)\}$ and \eqref{conv_pi} we get that 
\[
\sup_{a\in A(t), \ w\in\tilde\K} \{ p(t) \cdot \g(t,a,w)\} \leq 0 \qquad \forall t\in\S.
\] 
To obtain the maximality condition \eqref{maxham_G3_aux} on the whole interval $[0,T]$,  it suffices to notice that $\S$ is a dense set in $[0,T]$,  the function $t \mapsto \sup_{a\in A(t), \ w\in\tilde\K} \{ p(t) \cdot \g(t,a,w)\}$ is continuous thanks to (H4),  and  that $p(\cdot)$ is an absolutely continuous function.

\vsm

In order to conclude the proof, it remains to prove \eqref{maxham_G4}, that for the auxiliary control system reduces to
\bel{maxham_G4_aux}
p(t) \cdot \g(t,\bar\alpha(t),\bar\w(t)) =0 \qquad \text{$\bar\nu$-a.e. $t\in[0,T]$.}
\eeq
Taking $a=\alpha_i(t)$, $w=\w_i(t)$, and $m=0$,  from \eqref{maxham_ek} we obtain
$$
p_i(t)\cdot \g(t,\alpha_i(t),\w_i(t)) m_i(t) \geq -\eps_i m_i(t) \qquad \text{a.e. $t\in[0,T]$},
$$
from which it follows that
\bel{pi_Gi}
p_i(t) \cdot \g(t,\alpha_i(t),\w_i(t)) \geq -\eps_i \qquad \text{ for a.e. $t\in\{\tau\in[0,T] \text{ : } m_i(\tau)>0\}$.}
\eeq
Moreover, by \eqref{conv_G} and \eqref{conv_pi} we deduce that 
\[
p_i(t)\cdot \g(t,\alpha_i(t),\w_i(t)) m_i(t)dt \weak p(t) \cdot \g(t,\bar\alpha(t),\bar\w(t)) \bar\nu(dt).
\] 
Therefore, Lemma \ref{lemma_misure}, (iii), implies that, restricting attention to a suitable subsequence, there exists a subset $\hat\E\subset[0,T]$ of full measure, containing $T$,  such that
\bel{ultima_conv}
\lim_i \int_{[0,t]}p_i(s)\cdot \g(s,\alpha_i(s),\w_i(s)) m_i(s)ds = \int_{[0,t]}p(s) \cdot \g(s,\bar\alpha(s),\bar w(s)) \bar\nu(ds) \quad \forall t\in\hat\E.
\eeq
By \eqref{pi_Gi} and the fact that $m_i$ takes values in $\R_{\geq0}$, for any $t\in\hat\E$ we have
\[
\begin{split}
&\int_{[0,t]}p_i(s)\cdot \g(s,\alpha_i(s),\w_i(s)) m_i(s)ds \\ 
&\quad\qquad= \int_{[0,t]\cap \{\tau \text{ : }m_i(\tau)>0\}}p_i(s)\cdot \g(s,\alpha_i(s),\w_i(s)) m_i(s)ds \geq -\eps_i\int_{[0,t]} m_i(s) ds.
\end{split}
\]
Recalling that $(m_i)_i$ is uniformly bounded in $L^1$ by \eqref{mu_ek}, from the above relation and \eqref{ultima_conv} we deduce $\int_{[0,t]} p(s) \cdot \g(s,\bar\alpha(s),\bar\w(s)) \bar\nu(ds) \geq 0$ for any $t\in\hat\E$. But this implies 
$$
\int_{B} p(t) \cdot \g(t,\bar\alpha(t),\bar\w(t)) \bar\nu(dt) \geq 0 \qquad  \forall B\in\BB, 
$$
as the family of subsets $([0,t])_{t\in\hat\E}$ generates $\BB$. From this last relation it follows that 
$$
p(t)\cdot \g(t,\bar\alpha(t),\bar\w(t))\geq 0\qquad\text{$\bar\nu$-a.e. $t\in[0,T]$.} 
$$
The previous relation, condition \eqref{maxham_G3_aux} and the fact that $\bar\alpha(t)\in A(t)$ $\bar\nu$-a.e. $t\in[0,T]$ and $\bar\w(t)\in\tilde\K$  $\bar\nu$-a.e. $t\in[0,T]$, imply \eqref{maxham_G4_aux}. 
\qed

%
%
%
%

\section*{Appendix}

\begin{proof}[Proof of Lemma  \ref{lemma_convergenza}]
For any integer $i\geq0$ consider the measure $\tilde\gamma_i$ on the Borel subsets of $[-h,T]$ defined by $\tilde\gamma_i(B)=0$ for any Borel subset $B\subseteq[-h,0]$ and $\tilde\gamma_i\equiv\gamma_i$ on the Borel subsets of $]0,T]$. Hence, consider the function $\tilde z_0:[-h,T]\to\R^n$, absolutely continuous on $[0,T]$,  defined by $\tilde z_0(t):=z_0(t) -\int_{[0,t]}\tilde\gamma_0(ds)$ for any $t\in[-h,T]$, that turns out to satisfy the following   delayed differential equation
\[
\begin{cases}
\dot{\tilde z}_0(t) = \varphi_0\left( t, \left\{ \tilde z_0(t-h_k) +\int_{[0,t-h_k]} \tilde \gamma_0(ds) \right\} \right), \quad  t\in]0,T] \\
\tilde z_0(t) = \eta(t) \quad \text{a.e. } t\in [-h,0[, \qquad \tilde z_0(0)=\xi_0.
\end{cases}
\]
Hence, let $M>0$ be such that \footnote{Here we use the fact that $\eta(\cdot)$ is bounded.}
\[
\|\tilde z_0(t)\|_{L^\infty([-h,T])} + \| \gamma_i\|_{C^*(\R^k)} \leq \frac{M}{2(N+1)} 
\]
for any integer $i\geq 0$. Let $\tilde L_M\in L^1([0,T],\R_{\geq0})$ be as in \eqref{lipschitz_phi} and notice that \eqref{lipschitz_phi} and \eqref{bound_phi} are valid for $i=0$ as well, due to \eqref{conv_phi}. Now, for any integer $i\geq1$, consider the function $\beta_i:[0,T]\to\R_{\geq0}$ defined as
\[
\beta_i(t) := \left| \dot{\tilde z}_0(t) - \varphi_i\left( t, \left\{ \tilde z_0(t-h_k)  + \int_{[0,t-h_k]} \tilde\gamma_i(ds) \right\} \right) \right|
\]
and the set $\E_i$ given by 
\[
\E_i := \{ t \in[0,T] \text{ : } \varphi_i(t,\{x_k\}) = \varphi_0(t, \{z_k\}) \ \forall \{z_k\}\in(\R^n)^{N+1} \}.
\] 
Thus, for a.e. $t\in\E_i$,  we get
\[
\begin{array}{l}
\beta_i(t) = \Big| \varphi_0\big( t, \big\{ \tilde z_0(t-h_k) +\int_{[0,t-h_k]} \tilde \gamma_0(ds) \big\} \big) \\[1.2ex]
\quad - \varphi_0\big( t, \big\{ \tilde z_0(t-h_k) + \int_{[0,t-h_k]} \tilde\gamma_i(ds) \big\} \big) \Big| \leq \tilde L_M(t)\big| \big\{\int_{[0,t-h_k]}\big[\gamma_i(ds) -\gamma_0(ds) \big] \big\}\big|.
\end{array}
\]
Instead, for a.e. $t\in[0,T]\setminus\E_i$ we have
\[
\begin{array}{l}
\beta_i(t) \leq  |\dot{\tilde z}_0(t)| + \big| \varphi_i\big( t, \big\{ \tilde z_0(t-h_k) + \int_{[0,t-h_k]} \tilde\gamma_0(ds) \big\} \big)\big| \\[1.2ex]
\quad+\big| \varphi_i\big( t, \big\{ \tilde z_0(t-h_k) + \int_{[0,t-h_k]} \tilde\gamma_i(ds) \big\} \big) - \varphi_i\big( t, \big\{ \tilde z_0(t-h_k) + \int_{[0,t-h_k]} \tilde\gamma_0(ds) \big\} \big) \big| \\[1.2ex]
\ \qquad\leq 2\tilde c(t) + \tilde L_M(t) \big|\big\{ \int_{[0,t-h_k]} \big[\gamma_i(ds) -\gamma_0(ds) \big] \big\}\big|.
\end{array}
\]
Since  $\left| \int_{[0,t]} \left[\gamma_i(ds) -\gamma_0(ds) \right] \right|$ is uniformly bounded in $[0,T]$ and  converges to 0 for   all $t$ in the subset $\E\subseteq[0,T]$ of full measure and containing $t=T$,
 by applying the dominated convergence theorem and by \eqref{conv_phi} we obtain
\bel{conv_beta}
\lim_i \int_0^T \beta_i(t)\,dt =0.
\eeq
Thanks to \eqref{conv_beta} we can apply the Filippov Theorem for delayed systems \cite[Thm. 4.1]{VB} in order to deduce that for any integer $i$ sufficiently large there exists a function $\tilde z_i\in W^{1,1}([0,T],\R^n)$ that solves \footnote{
The function $(t,\{z_k\})\mapsto \tilde\varphi_i(t,\{z_k\}):= \varphi_i\big(t, \big\{ z_k + \int_{[0,t-h_k]} \tilde\gamma_i(ds) \big\} \big)$ is measurable in the $t$ variable. Moreover, $\{z_k\}\in \{\tilde z_0(t-h_k)\} + \frac{M}{2} \B_{n(N+1)}$ implies that $\big\{ z_k + \int_{[0,t-h_k]} \tilde\gamma_i(ds) \big\}\in M \B_{n(N+1)}$, so that $\tilde\varphi_i$ is $\tilde L_M(\cdot)$-Lipschitz continuous in the $\{z_k\}$ variables in the $\frac{M}{2}$-tube around $\{\tilde z_0(t-h_k)\}$.
}
\bel{equazione_zi}
\begin{cases}
\dot{\tilde z}_i(t) = \varphi_i\left( t, \left\{ \tilde z_i(t-h_k) + \int_{[0,t-h_k]}\tilde\gamma_i(ds) \right\}\right),  \qquad  t\in]0,T] \\
\tilde z_i(t) =\eta(t) \quad \text{a.e. }   t\in[-h,0[, \qquad \tilde z_i(0)= \xi_i.
\end{cases}
\eeq
and that satisfies
\bel{conv_zeta} 
\begin{split}
\| \tilde z_i - \tilde z_0 \|_{L^\infty([0,T])} &\leq |\xi_i-\xi_0| +\int_0^T |\dot{\tilde z}_i(t) - \dot{\tilde z}_0(t)|\,dt \\
& \leq e^{(N+1)\int_0^T \tilde L_M(t)dt} \left( |\xi_i-\xi_0| + \int_0^T \beta_i(t)dt \right).
\end{split}
\eeq
For any integer $i\geq 1$ we define the function $z_i:[-h,T]\to\R^n$ to be $z_i(t)=\tilde z_i(t)+\int_{[0,t]} \gamma_i(ds)$ for any $t\in]0,T]$, $z_i(0)=\xi_i$, and $z_i(t)=\eta(t)$ for $t\in[-h,0[$. In view of \eqref{equazione_zi} we deduce that $z_i$ belongs to $BV([-h,T],\R^n)$ and is a solution to \eqref{eq_z}. Furthermore, by \eqref{conv_zeta}, the fact that $\xi_i\to\xi_0$, and the hypothesis on the convergence of the $\gamma_i$, we immediately deduce that $dz_i\weak dz_0$ and $z_i(t)-\int_{[0,t]} \gamma_i(ds) \to z_0(t) - \int_{[0,t]} \gamma_0(ds)$ uniformly in $[0,T]$, so that $z_i(t)\to z(t)$ for all $t\in\E$.
\end{proof}

\begin{proof}[Proof of Lemma  \ref{lemma_misure}, (ii), (iii)] Let $(\gamma_i,|\gamma_i|)\weak(\gamma,\lambda)$ for some $\lambda\in C^\oplus$  and let $\Psi\in C([0,T],\R^{n\times k})$ be as in (ii).  Let $\tilde\E\subseteq[0,T]$ be the subset of points $t$ such that $\lambda(\{t\})=0$. Clearly, $[0,T]\setminus\tilde\E$ is a countable set.  For every $t\in\E:= \tilde\E\cup\{T\}$ and any integer $i\ge1$, consider  
$$
\int_{[0,t]}\Psi(s)\cdot\gamma_i(ds)=\left(\int_{[0,t]}\Psi_1(s)\cdot\gamma_i(ds),\dots,\int_{[0,t]}\Psi_n(s)\cdot\gamma_i(ds)\right),
$$
where, for every $l=1,\dots,n$,   $\Psi_l(s)$ is the row vector $(\Psi_{l1}(s),\dots, \Psi_{lk}(s))$.  Since this integral  converges  if and only if each component converges, it is sufficient  to prove that  
$$ 
\int_{[0,t]}\Psi_l(s)\cdot\gamma_i(ds)\to \int_{[0,t]}\Psi_l(s)\cdot\gamma(ds) \quad \text{as $i\to\infty$}
$$
for each $l=1,\dots,n$. This is equivalent to show that \footnote{Actually, according to our definition \eqref{def_integrals},  for each $i$, $\int_{[0,t]}\Psi_l(s)\cdot\gamma_i(ds)=\int_{[0,t]}\Psi_l(s)\cdot\w_i(s)|\gamma_i|(ds)=\int_{[0,t]}\sum_{j=1}^k\Psi_{lj}(s)\w_i^j(s)|\gamma_i|(ds)$, where $\w_i$ is the Radon-Nicodym derivative of the measure $\gamma_i$ w.r.t. $|\gamma_i|$. Since, however, $[0,T]\ni t\mapsto \Psi_{l}(t)$ is continuous (and bounded), all components  $\Psi_{lj}(s)$ are $\gamma_i^j$ integrable and  $\int_{[0,t]}\sum_{j=1}^k\Psi_{lj}(s)\w_i^j(s)|\gamma_i|(ds)=\sum_{j=1}^k\int_{[0,t]}\Psi_{lj}(s)\gamma^j_i(ds)$. }
$$
\sum_{j=1}^k\int_{[0,T]}\Psi_{lj}(s)\chi_{_{[0,t]}}(s)\gamma^j_i(ds)\to \sum_{j=1}^k\int_{[0,T]}\Psi_{lj}(s)\chi_{_{[0,t]}}(s)\gamma^j(ds) \quad \text{as $i\to\infty$},
$$
which is certainly true if, for every $j=1,\dots,k$, we have
\bel{lemma3_th}
\int_{[0,T]}\Psi_{lj}(s)\chi_{_{[0,t]}}(s)\gamma^j_i(ds)\to \int_{[0,T]}\Psi_{lj}(s)\chi_{_{[0,t]}}(s)\gamma^j(ds) \quad \text{as $i\to\infty$}.
\eeq
Note that, for $t=T$, the thesis follows from the definition of weak convergence, in view of the continuity of $\Psi$. Thus, let $t\in\tilde\E$. At this point, to complete the proof of (ii)  we can appeal to  \cite[Prop. 1.62]{AFP},     observing that, for each pair $l$, $j$,  the scalar function $ s\mapsto \Psi_{lj}(s)\chi_{_{[0,t]}}(s)$ is a bounded, Borel measurable function such that its unique discontinuity point, $t$, has    $\lambda(\{t\})=0$.  

  In order to prove statement (iii), notice that, since $\gamma_i\weak\gamma$, if for each $j=1,\dots,k$ we consider the Jordan decomposition $\gamma^j_i=(\gamma^j_i)^+-(\gamma^j_i)^-$,  by uniform boundedness,  possibly extracting a subsequence, we obtain that  $(\gamma^j_{i_l})^+\weak\lambda^{j+}$ and $(\gamma^j_{i_l})^-\weak\lambda^{j-}$,  for some positive, finite  measures $\lambda^{j+}$, $\lambda^{j-}$. Then, 
$$
|\gamma^j_{i_l}|=(\gamma^j_{i_l})^++(\gamma^j_{i_l})^-\weak \lambda^j:=\lambda^{j+}+\lambda^{j-},
$$
which implies that $|\gamma_{i_l}|=\sum_{j=1}^k|\gamma^j_{i_l}|\weak \lambda:=\sum_{j=1}^k\lambda^j\in C^\oplus$, 
so that, along this subsequence, we have 
$(\gamma_{i_l},|\gamma_{i_l}|)\weak(\gamma,\lambda)$.
%
\end{proof}
%

\section*{Acknowledgments} 
We would like to thank Prof. Richard Vinter for suggesting the problem and for helpful discussions.

\end{document}